\documentclass[11pt]{amsart}
\usepackage{amsmath,amssymb,amsfonts}
\usepackage{enumerate}
\usepackage{amscd, amssymb, latexsym, amsmath, amscd}
\usepackage[all]{xy}
\usepackage{pb-diagram}
\usepackage{multicol,lipsum}

\newtheorem{thm}[equation]{Theorem}
\newtheorem{prop}[equation]{Proposition}
\newtheorem{cor}[equation]{Corollary}

\newtheorem{lemma}[equation]{Lemma}

\numberwithin{equation}{section}

\newcommand{\Q}{\mathbb Q}
\newcommand{\Z}{\mathbb Z}
\newcommand{\R}{\mathbb R}
\newcommand{\C}{\mathbb C}

\newcommand{\G}{\mathfrak G}

\newcommand{\A}{\mathbb A}

\def\Hom{{\rm Hom}}

\def\G{{\rm G}}
\def\SL{{\rm SL}}

\def\GSp{{\rm GSp}}

\def\Sp{{\rm Sp}}

\def\U{{\rm U}}

\def\GL{{\rm GL}}

\def\SO{{\rm SO}}

\def\Sp{{\rm Sp}}

\def\O{{\rm O}}

\usepackage[all]{xy}

\def\A{{\mathbb A}}

\def\R{{\mathbb R}}

\def\Z{{\mathbb Z}}

\def\C{{\bf C}}

\def\C{{\mathbb C}}

\def\G{{\mathbb G}}

\usepackage{picture}

\title[Globalization over Function Fields]{Globalization of supercuspidal representations over function fields and applications}
\author{Wee Teck Gan}
\address{Department of Mathematics, National University of Singapore, Block S17, 10 Lower Kent Ridge Road, Singapore 119076}
\email{matgwt@nus.edu.sg}

\author{Luis Lomel\'i}
\address{Max-Planck-Institut fur Mathematik,  Vivatsgasse 7, 53111 Bonn, Germany}
\email{lomeli@mpim-bonn.mpg.de}

\begin{document}
\maketitle

\begin{abstract}
Let $H$ be a connected reductive group defined over a non-archimedean local field $F$ of characteristic $p>0$. Using Poincar\'e series, 
we globalize supercuspidal representations of $H_F$ in such a way that we have control over ramification at all other places, and  
such that the notion of distinction with respect to a unipotent subgroup (indeed more general subgroups) is preserved.
In combination with the work of Vincent Lafforgue on the global Langlands correspondence, 
we present some applications, such as the stability of Langlands-Shahidi $\gamma$-factors and the local Langlands correspondence for classical groups.
\end{abstract}

\section{\bf Statement of Results}
In this  paper, we present a useful globalization result for supercuspidal representations over a non-archimedean local field of characteristic $p>0$.
  
\vskip 5pt
\begin{thm}  \label{T:main}
Suppose we are given the following data:
\vskip 5pt

\begin{itemize}
\item $k = \mathbb{F}_q(Y)$ is the global function field of an absolutely irreducible smooth projective curve $Y$ over a finite field $\mathbb{F}_q$, with associated ring of adeles $\A$;
\vskip 5pt

\item  $S_0$ is a nonempty finite set of places of $k$;
\vskip 5pt

\item $H$ is a smooth connected reductive group over $k$ with $Z$ the identity component of its center;
\vskip 5pt

\item $N \subset H$ is a (possibly trivial) smooth connected $k$-split unipotent subgroup over $k$;
\vskip 5pt

\item $\chi = \prod_v \chi_v :  N(\A)  \longrightarrow \C^{\times}$ is a (possibly trivial) unitary character trivial on $N(k)$; 
\vskip 5pt

 \item  $\omega = \prod_v  \omega_v$ is a character of $Z(k) \backslash Z(\A)$;  
\vskip 5pt

\item for each $v_0 \in S_0$, $\pi_{v_0}$ is a supercuspidal  representation of $H(k_{v_0})$  which is $(Z(k_{v_0}),\omega_{v_0})$- and  $(N(k_{v_0}), \chi_{v_0})$-distinguished, 
i.e. satisfying 
\[  \Hom_{Z(k_{v_0}) \cdot N(k_{v_0})}(\pi_{v_0},  \omega_{v_0} \otimes \chi_{v_0}) \ne 0. \]
\vskip 5pt

\end{itemize}
\vskip 5pt

\noindent  
 Then there exists a cuspidal representation $\Pi$ of $H(\A)$ satisfying:
\vskip 5pt

\begin{itemize}
\item[(i)]  for all $v_0 \in S_0$, $\Pi_{v_0}  \cong \pi_{v_0}$;
\vskip 5pt

\item[(ii)]  for all $v \notin S_0$,  $\Pi_v$ is a constituent of a principal series representation induced from a minimal parabolic subgroup of $H_v$ and whose restriction to the derived group $H_v^{der}$ has depth $0$; 
\vskip 5pt

 \item[(iii)]  $\Pi$ has central character $\omega$ and nonzero automorphic $(N, \chi)$-period.
\end{itemize}
\end{thm}

  \vskip 5pt
 We make a few remarks:
 \vskip 5pt
 \begin{itemize}
 \item[-]   If $N$ is trivial, then the local condition and global conclusion about $(N,\chi)$ are vacuous and thus $(N,\chi)$ can be suppressed.  
 \vskip 5pt
 
 \item[-] If $H$ is quasi-split over $k$, $N$ is a maximal unipotent subgroup of $H$ and $\chi$ is a generic character, then (iii) says that $\Pi$ is globally $\chi$-generic. Moreover, (ii) implies that for $v \notin S_0$, $\Pi_v$ is induced from  the Borel subgroup.
 \vskip 5pt
 
 \item[-] One has to be careful in working with unipotent subgroups over non-perfect fields, even if one is working with smooth connected groups, as these may not be $k$-split (i.e. successive extensions of the additive group $\mathbb{G}_a$). For these subtleties, the reader can consult \cite[Chap. 5]{O} or \cite[Appendix B]{CGP}. In this paper, we shall only consider smooth connected $k$-split unipotent groups and these are isomorphic to affine spaces as algebraic varieties.
 We  shall abbreviate the terminology by simply referring to these as unipotent groups and this abbreviation will be used without further comment.
   \end{itemize}

\vskip 5pt
  We should mention that the cuspidal representation in Theorem \ref{T:main}  is constructed by means of Poincar\'e series. Such globalization results were first proved by Henniart [H] and extended by Vigneras \cite{V} (over arbitrary global fields) and Shahidi \cite{Sh2} (over number fields). A recent preprint of Moy-Mui\'c \cite{MM} further refines this series of results over number fields, allowing one to globalize nonsupercuspidal representations (under certain hypotheses). There is also an analogous globalization result due to S.W. Shin over totally real fields \cite{Shin} proved using the Arthur trace formula. However, in all these versions, one loses control of the local component of the cuspidal representation at one place of $k$, typically an archimedean place. Our Theorem, on the other hand, gives rather good control at all places.  The proof of our theorem is inspired by \cite[Theorem 3.3]{HL3}, which is a special case of Theorem~\ref{T:main} in the context of $\GL_n$. The slight improvement over the treatment in \cite[Theorem 3.3]{HL3}  is that we make no use of the fact/hypothesis that supercuspidal representations can be constructed by compact induction. For the case of generic representations of quasi-split reductive groups mentioned in the remark above, a proof can also be found in  \cite[\S4.1] {L2}.  \vskip 5pt

The following corollary  of Theorem \ref{T:main} is useful  in practice.
\vskip 5pt
 
\begin{cor}  \label{T:main2}
Let $F$ be a local field of characteristic $p >0$ and let $H_F$ be a connected reductive group over $F$ with $Z_F$ the identity component of its center and $N_F$ the unipotent radical of a parabolic $F$-subgroup $P_F = M_F \cdot N_F$ (so $N_F$ is possibly trivial here). Assume that  $\chi_F$ is a unitary character of $N_F$ which lies in an open $M_F$-orbit.
\vskip 5pt

Suppose that $\pi_1, \ldots ,\pi_a$ is a collection of supercuspidal representations of $H_F$ which have the same central character under $Z_F$ and  which are  $(N_F, \chi_F)$-distinguished (and hence distinguished with respect to any character in the same $M_F$-orbit as $\chi_F$). Then there exist 
\vskip 5pt
\begin{itemize}
\item a global function field $k$, with a finite set $\left\{v_1,\ldots,v_a\right\}$ of places and isomorphisms $k_{v_i} \cong F$; 
\vskip 5pt

\item  a connected reductive $k$-group $H$ with isomorphisms $H_{v_i} \cong H_F$,
 containing a parabolic $k$-subgroup $P = M \cdot N$ such that $P_{v_i} \cong P_F$;
 \vskip 5pt
 
 \item a unitary character $\chi$ of $N(\A)$ trivial on $N(k)$ such that $\chi_{v_i}$ and  $\chi_F$   lie in the same $M_F$-orbit for each $i$;
\vskip 5pt

\item a cuspidal automorphic representation $\Pi$ of $H(\A)$ which is globally $(N, \chi)$-distinguished, with  $\Pi_{v_i}  \cong \pi_i$ for $i=1,\ldots,a$, and with $\Pi_v$ contained in a  principal series representation induced from a minimal parabolic subgroup for all other $v$, such that $\Pi_v$ is of depth $0$ when restricted to $H_v^{der}$. 
\end{itemize}
\end{cor}
\vskip 5pt

  The main point of the corollary is that only local data is given, and so one needs to globalize several objects (such as the field, the various groups and the various characters) before one is in a position to apply Theorem \ref{T:main}. 
  Moreover, if we set $W_F = \Hom_F(N_F, \G_a)$ and fix a nontrivial character $\psi_F$ of $F$, then composition with $\psi_F$ identifies the $F$-vector space $W_F$ with the set of unitary characters of $N_F$.  Thus $\chi_F$ 
  is an element of $W_F$ and we are requiring in the corollary that its $M_F$-orbit is Zariski open in $W_F$.
  \vskip 5pt

Using the globalization of supercuspidal representations such as given by Theorem \ref{T:main} and Corollary \ref{T:main2}, the second author has completed the Langlands-Shahidi theory in positive characteristic \cite{L1,L2},  following Shahidi's work \cite{Sh2} in characteristic zero. In particular, one has a characterization of the Langlands-Shahidi gamma factors for generic representations of quasi-split groups over function fields by the usual properties: 
multiplicativity, compatibility with class field theory in the case of tori and global functional equations. Special cases of this characterization over function fields were shown in 
\cite{HL1,HL2,HL3,GL}. We refer the reader to \cite{L2} for the general results.
\vskip 5pt

When one combines Theorem \ref{T:main}  with the Langlands-Shahidi theory and  the recent work \cite{La} of V. Lafforgue on the global Langlands correspondence over function fields, one can obtain further applications.
Let us highlight some of these here:
 \vskip 5pt
 
\begin{itemize}
\item  In Theorem \ref{T:stability}, we show the stability of general Langlands-Shahidi gamma factors in positive characteristic.  
\vskip 5pt

\item  In Theorem \ref{T:plan},  we express the Plancherel measure (associated with parabolic induction) in terms of Galois theoretic gamma factors.   

\vskip 5pt

\item Building upon these results, and appealing to the work of V. Lafforgue, L. Lafforgue, Deligne and others, we can attach local L-parameters to supercuspidal representations of quasi-split classical groups under a working hypothesis (see \S \ref{SS:hypo} and Theorem \ref{T:L}).  
\vskip 5pt

\item Our results on Plancherel measure (together with a result of Silberger) also allow us to verify the basic assumption (BA) in the work of Moeglin-Tadi\'c on the classification of discrete series representations of quasi-split classical groups in terms of supercuspidal ones. As a consequence, we can extend the local Langlands correspondence for supercuspidal representations obtained above to all discrete series representations, and then to all irreducible smooth representations by Langlands classification; see Theorem \ref{T:L2}. 
\end{itemize}

\vskip 5pt

We will discuss these various applications in \S \ref{S:applications}, \S \ref{S:plan} and  \S \ref{S:LLC}  respectively.  
In  \S \ref{S:LLC2}, we discuss another approach to extending the local Langlands correspondence of classical groups from supercuspidal representations to discrete series representations, using a (conjectural) simple form of the Arthur-Selberg trace formula.  
The main point is to globalize a discrete series representation (in the style of our main theorem), but the Poincar\'e series argument does not apply, which necessitates the use of the trace formula. We hope that this application will provide some impetus for the systematic development of the local theory of invariant harmonic analysis and the global theory of the trace formula in positive characteristic.

\vskip 10pt

The applications above are all obtained via  a global-to-local argument, using a globalization result of the type in the theorem. As we mentioned above, in such globalization, one often loses control at some place of $k$.  In characteristic 0, one sacrifices the archimedean places, and so one needs to have already established the desired theorem at archimedean places by purely local means. The local proof of the archimedean theorem could be highly nontrivial but is thankfully 
 more attainable than the nonarchimedean one. In the context of the Langlands-Shahidi theory in characteristic 0, this archimedean input was provided by Shahidi \cite{Sh1}. The main stumbling block preventing the development of the Langlands-Shahidi theory  in positive characteristic was the constraint  that  one cannot sacrifice any local place, since the desired result is not known at any place. Hence, it would appear that this situation is one of the few instances where having some archimedean places is a blessing instead of a curse, which is quite contrary to the general principle that function  fields are easier to handle than number fields because of a rich underlying geometry. Another such instance is the state of the Arthur-Selberg trace formula. With Theorem \ref{T:main}, however,   we remove the previous constraint  and there is no longer a need to sacrifice any place over a function field. So the globalization of supercuspidal representations  over function fields turns out to be easier to handle than number fields.

\vskip 5pt

Finally, we show a variant of Theorem \ref{T:main}, which is a refinement in positive characteristic of a theorem of D. Prasad and R. Schulze-Pillot \cite[Theorem 4.1]{PSP} on globalizing supercuspidal representations  that are distinguished with respect to a given closed algebraic subgroup (which is not necessarily unipotent):
\vskip 5pt

\begin{thm} \label{T:main3}
Suppose we are given the following data:
\vskip 5pt

\begin{itemize}
\item $k = \mathbb{F}_q(Y)$ is the global function field of an absolutely irreducible smooth projective curve $Y$ over a finite field $\mathbb{F}_q$, with associated ring of adeles $\A$;
\vskip 5pt

\item  $S_0$ is a nonempty finite set of places of $k$;
\vskip 5pt

\item $H$ is a connected reductive group over $k$, with $Z$ the identity component of its center;
\vskip 5pt

\item $R \subset H$ is a closed algebraic $k$-subgroup containing $Z$ and such that $R/Z$ has no nontrivial $k$-rational characters;
\vskip 5pt

\item $\chi = \prod_v \chi_v :  R(\A)  \longrightarrow \C^{\times}$ is a (possibly trivial) unitary character trivial on $R(k)$; 
\vskip 5pt

\item for each $v_0 \in S_0$, $\pi_{v_0}$ is a supercuspidal  representation of $H(k_{v_0})$  which is $(R(k_{v_0}),  \chi_{v_0})$-distinguished, i.e.
satisfying $\Hom_{R(k_{v_0})}(\pi_{v_0}, \chi_{v_0}) \ne 0$.
\vskip 5pt

\end{itemize}
\vskip 5pt

We make the following two technical assumptions:
\vskip 5pt

\begin{itemize}
\item[(a)]  there is a semisimple algebraic representation $\iota:  H \longrightarrow \GL(V)$ defined over $k$ such that $R$ is the stabilizer of a vector $x_0 \in V(k)$ and ${\rm Lie}(R)$ is the infinitesimal stabilizer of $x_0$;

\item[(b)]  for all places $v \notin S_0$, there exists an Iwahori subgroup $I^{der}_v$ of $H^{der}(k_v)$ with pro-$p$ radical $J^{der}_v$ such that $\chi_v$ is trivial on $R(k_v) \cap J^{der}_v$. 
\end{itemize}
\noindent    Then there exists a cuspidal representation $\Pi$ of $H(\A)$ satisfying:
\vskip 5pt

\begin{itemize}
\item[(i)]  for all $v_0 \in S_0$, $\Pi_{v_0}  \cong \pi_{v_0}$;
\vskip 5pt

\item[(ii)]  for all $v \notin S_0$,  $\Pi_v$ belongs to a  principal series representations induced from a minimal parabolic subgroup and  has depth $0$ when restricted to $H_v^{der}$;
\vskip 5pt

\item[(iii)]  $\Pi$ has nonzero automorphic $(R, \chi)$-period.
\end{itemize}

\end{thm}
 \vskip 5pt
 
\vskip 5pt

We make a couple of remarks about the technical conditions (a) and (b). By a well-known theorem of Chevalley \cite[Theorem 5.1 and \S 5.5]{B}, given any closed algebraic subgroup $R$ of $H$ as in Theorem \ref{T:main3}, there is an algebraic representation $\iota:  H \longrightarrow \GL(V)$ such that $R$ is the stabilizer of  a vector $x_0 \in V(k)$ and ${\rm Lie}(R)$ is the infinitesimal stabilizer of $x_0$ (here we are using the hypothesis that $R/Z$ has no nontrivial rational characters).  As R. Beuzart-Plessis explained to us, this implies that $R \backslash H$ is quasi-affine and hence $Z(\A) R(k) \backslash R(\A)$ is a closed subset of $Z(\A) H(k) \backslash H(\A)$. Since cusp forms on $H(k)\backslash H(\A)$ are compactly supported modulo $Z(\A)$, the automorphic $(R,\chi)$-period is absolutely convergent and it makes sense to consider it. 
However, in the above, there is no guarantee that $\iota$ is a semisimple representation, and this semisimplicity is of course an issue in characteristic $p >0$. For our proof of Theorem \ref{T:main3}, we need $\iota$ to be semisimple and this explains the technical condition (a). 
For  technical assumption (b) in Theorem \ref{T:main3}, note that the requirement  is satisfied automatically for almost all places $v$, and is satisfied for all $v \notin S_0$ if $\chi$ is the trivial character.
  
\vskip 10pt
\noindent{\bf Acknowledgments}: 
We thank C W. Chin, B. Conrad, G. Harder, G. Henniart, D. Prasad, G. Savin, F. Shahidi and L. Zhang for useful discussions during the course of this work. We are grateful to V. Lafforgue for his comments on the first draft of this paper.   We would also like to acknowledge useful discussions with B. H. Gross, E. Lapid and R. Beuzart-Plessis which greatly clarify for us  the material of \S \ref{S:LLC2}.
 Finally, we thank the referee of our paper for his many pertinent comments and suggestions which improved the accuracy and exposition of the paper.
  \vskip 5pt
  
The first author is partially supported by a Singapore government MOE Tier 2 grant R-146-000-175-112.
 This paper is based upon work supported by the National Science Foundation under Grant No.~0932078 000 while the authors were in residence at the Mathematical Sciences Research Institute in Berkeley, California, during the Fall 2014 semester. We thank MSRI for providing excellent working conditions. The second author would like to thank the MaxPlanck Institute for Mathematics for its hospitality during the year 2015.

   \vskip 10pt

\section{\bf Proof of Theorem \ref{T:main}}

In this section, we give the  proof of Theorem \ref{T:main}. 
We first assume that $H$ is semisimple. Even though a uniform argument can be given, we shall first deal with the case when $H$ is split, as it is notationally cleaner and conceptually 
simpler.  Throughout, let  $C_N = \prod_v C_{N,v} \subset N(\A)$ be a compact subgroup which projects surjectively  onto $N(k) \backslash N(\A)$. Note that $C_N$ exists because $N(k) \backslash N(\A)$ is compact and the totally disconnected group $N(\A)$ contains arbitrarily large open compact subgroups.

\vskip 5pt

\subsection{\bf Split semisimple case.}  \label{SS:split-ss}
With $H$ a split semisimple group, choose an inclusion
\[  \iota:  H  \longrightarrow \SL_n  \subset \GL_n   \]
over $k$, which allows us to identify $H$ as a closed subgroup of $\GL_n$. Then we have the pullback of the $n^2$ coordinate function $x_{ij}$ on $H$.  
Without loss of generality,  we may assume that the intersection of  $H$ with the upper (respectively lower) triangular Borel subgroup of $\GL_n$ is a Borel subgroup $B = T \cdot U$ 
(respectively $\overline{B}  = T \cdot \overline{U}$) of $H$, and that
 $N \subset U$.  Indeed, since $N$ is unipotent, we may choose a Borel subgroup $B = T \cdot U$ of $H$ such that $N \subset U$. Then $\iota(B) \subset \SL_n$ is a connected solvable subgroup and one may  conjugate the pair $\iota(T)$ and  $\iota(B)$ to lie inside the diagonal torus and the standard Borel subgroup of upper triangular matrices in $\SL_n$ respectively, from which it follows that $\iota(\overline{U})$ is conjugated into the subgroup of lower triangular unipotent matrices.
 As an affine space, we may write $U  = N \times N'$ with $N$ and $N'$ affine subspaces.
\vskip 5pt

Let $\mathcal{O}_{S_0}$ denote the ring of $S_0$-integers (i.e. the subring of elements of $k$ which have no poles outside $S_0$). Then the ``natural" $\mathcal{O}_{S_0}$ integral structure on 
$\GL_n$ induces one on $H$ and $N$. 
 Now let $S$ be a finite set of places of $k$ disjoint from $S_0$  such that for all $v \notin S \cup S_0$,
\vskip 5pt

\begin{itemize}
\item  the groups $H$, $B$, $T$ and $U$  are smooth over $\mathcal{O}_v$ and  $H(\mathcal{O}_v)  = H(k_v) \cap \GL_n(\mathcal{O}_v)$ is a 
hyperspecial maximal compact subgroup; 
\vskip 5pt
\item the intersection of the upper triangular and lower triangular Iwahori subgroups of $\GL_n(k_v)$ give 
Iwahori subgroups  of $H(k_v)$; we denote these by $I_v^+$ and $I_v^-$ respectively.

\vskip 5pt

\item   the decomposition $U = N \times N'$ is defined over $\mathcal{O}_v$, with $N$ and $N'$ smooth over $\mathcal{O}_v$;
\vskip 5pt

\item $C_{N,v} = N(\mathcal{O}_v)$ and $\chi_v$ is trivial when restricted to $N(\mathcal{O}_v)$.
 \end{itemize}
\noindent Note that the above conditions can be achieved when $S$ is large enough. For the first condition, see \cite[\S3.9]{T}. For the second condition, suppose that  $v$ is a place with associated residue field $\kappa_v$  such that  the first condition holds. Then one has a commutative diagram induced by the projection map $\mathcal{O}_v \rightarrow \kappa_v$:
\[  \begin{CD}
B(\mathcal{O}_v) @>>> H(\mathcal{O}_v)  @>>>   \GL_n(\mathcal{O}_v)  \\
@VVV @VVV   @VVV  \\
B(\kappa_v) @>>> H(\kappa_v)  @>>>  \GL_n(\kappa_v),  \end{CD} \]
where $B (\kappa_v)$ is a Borel subgroup of $H(\kappa_v)$ and is the intersection of $H(\kappa_v)$ with  the standard (upper triangular) Borel subgroup in $\GL_n( \kappa_v)$. Now recall from \cite[\S 3.7]{T} that the preimage  in $\GL_n(\mathcal{O}_v)$ of the standard Borel subgroup of   $\GL_n(\kappa_v)$ is the standard Iwahori subgroup of $\GL_n(k_v)$ and the   preimage in $H(\mathcal{O}_v)$ of $B(\kappa_v)$ is an Iwahori subgroup of $H(k_v)$. It follows from this that the second condition holds. 

 \vskip 5pt

Now fix an open compact subset $C_{S_0}$ of $H(k_{S_0}) = \prod_{v \in S_0} H(k_v)$ and some nonempty finite set of places $S_1 \cup S_2$ of $k$ disjoint from $S \cup S_0$. 
We are going to define an open compact subset $C = \prod_v C_v$ as follows:
\vskip 5pt

\begin{itemize}
\item we take
\[  \prod_{v \in S_0} C_v  = C_{S_0}. \]
\vskip 5pt
 \item for $v \in S$,  let $C_v$ be an Iwahori subgroup $I_v$ of $H(k_v)$ such that
$\chi_v$ restricted to $N(k_v) \cap I_v$ is trivial.   

\vskip 5pt
\item for $v \in S_1$, let $C_v$ be the pro-$p$ radical $J_v^+$ of $I_v^+$;
\vskip 5pt
\item for $v \in S_2$, let $C_v$ be the Iwahori subgroup $I_v^-$; 
\vskip 5pt
\item for all other places $v$,   let $C_v = K_v = H(\mathcal{O}_v)$.
\end{itemize}
\vskip 5pt

\noindent  Note that for the second condition above, the desired Iwahori subgroup always exists. Indeed, suppose one starts with  any Iwahori subgroup $I'_v$ stabilizing a chamber in the apartment associated to $T$ in the Bruhat-Tits building of $H$.  For $t \in T(k_v)$, the compact open subgroup $t I'_v t^{-1} \cap N(k_v)$ can be made arbitrarily small by taking $t$ sufficiently deep into the positive Weyl chamber, i.e. by ensuring that $|\alpha(t)|_v$ is sufficiently small for all positive roots of $H$ with respect to $(T, B)$. Since $\chi_v$ is smooth, it will be trivial on  $t I'_v t^{-1} \cap N(k_v)$ when the latter is sufficiently small, and one can take $I_v = t I'_v t^{-1}$.
\vskip 5pt

The following is a key lemma:
\vskip 5pt

\begin{lemma}  \label{L:key}
If $S_1$ and $S_2$ are sufficiently large, one has
\[  H(k)  \cap C \cdot C_N  \subset  N(k),  \]
with the intersection occurring in $H(\A)$. Indeed, one may take $S_1$ and $S_2$ to be singleton sets, each containing a place whose residue field is sufficiently large.
\end{lemma}
\vskip 5pt

\begin{proof}
We regard $\gamma \in H(k)$ as an element in $\GL_n(k)$, so that $\gamma$ is determined by its coordinates $x_{ij}(\gamma) \in k$.
 Consider $\gamma \in H(k) \cap C  \cdot C_N$.
  Away from the set $T= S \cup S_0 \cup S_1 \cup S_2$ of places,  
  \[  C^T \cdot C^T_N  = \prod_{v \notin T} K_v \]
    and so for $v \notin T$, $x_{ij}(\gamma) \in \mathcal{O}_v$.
  Hence $x_{ij}(\gamma)$ has no poles at the closed points of $Y$ outside $T$.
 We now consider the places in $T$:
\vskip 5pt

\begin{itemize}
\item At places $v \in S_0 \cup S$,  $x_{ij}(\gamma)$ has bounded orders of poles (determined by the compact sets $C_v \cdot C_{N,v}$). 
More precisely, there exists a positive integer $M$ (depending only on $C_{N_v}$, $\chi_v$ and $C_v$ for $v \in S \cup S_0$)  such that  for all $i, j$,  the order of poles of $x_{ij}(\gamma)$ is at most $M$ at all $v \in S_0 \cup S$.

\vskip 5pt
\item at places $v_1 \in S_1$, the condition that $\gamma \in C_{v_1}\cdot C_{N,{v_1}}$ implies that $x_{ij}(\gamma)$ vanishes at $v_1$ for all $i > j$, and 
$x_{ii}(\gamma)  -1$ vanishes at $v_1$ for all $i$.
 
 \vskip 5pt

\item   at places $v_2 \in S_2$,  the condition that $\gamma \in C_{v_2}  \cdot C_{N,{v_2}}$ implies at least that $x_{ij}(\gamma)$ has no poles at $v_2$.
 \end{itemize}
 \vskip 5pt

Since a principal divisor on $Y$ has degree $0$ (i.e. by the product formula), it is clear that if $S_1$ is sufficiently large, $x_{ij}(\gamma) = 0$ for all $i > j$ and $x_{ii}(\gamma)  =1$ for all $i$. 
In particular, any $\gamma \in H(k) \cap C \cdot C_N$ is strictly upper triangular and hence lies in $U(k)$.  In fact, one could take $S_1$ to contain only one place $v_1$ whose residue field (and hence the Galois orbit of points on  $Y(\mathbb{F}_q)$ associated to $v_1$) is  sufficiently large. This is possible since $Y(\overline{\mathbb{F}}_q)$ is infinite whereas $Y(\mathbb{F}_{q^n})$ is finite for each $n$. 
\vskip 5pt

Now we claim that if  $S_2$ is sufficiently large, the condition that $\gamma \in C_{v_2} \cdot C_{N,{v_2}}$  implies that $\gamma \in N(k)$. Indeed,
one may change coordinates on $U$, so that with respect to the new coordinates $y_{ij}$ ($i<j$), the subspace $N$ is defined by the vanishing of a subset $\Sigma$ of the $y_{ij}$'s.
For $v \in S_2$,  $U =  N \times N'$ is defined over $\mathcal{O}_v$. For $\gamma \in U(k)  \cap C_{v_2} \cdot C_{N,{v_2}}$, 
 one sees that $y_{ij}(\gamma)  \in \mathcal{O}_{v_2}$ for all $i< j$, and $y_{ij}(\gamma)$ vanishes at $v_2$ for $(i,j) \in \Sigma$.  Hence if $S_2$ is sufficiently large, $y_{ij} (\gamma)= 0$ for all $(i,j) \in \Sigma$, and we conclude that $\gamma \in N(k)$. As above, one could also have taken $S_2$ to consist of a single place $v_2$ whose residue field is sufficiently large.
 \end{proof}
\vskip 15pt

 We shall now define  a test function $f  = f_{S_1, S_2} = \prod_v  f_v \in C^{\infty}_c(H(\A))$ as follows:
\vskip 5pt
\begin{itemize}
\item for $v_0 \in S_0$, take
\[  f_{v_0}(h)  =  \langle w_{v_0}^{\vee}, h \cdot w_{v_0} \rangle \]
 to be the (compactly supported) matrix coefficient of $\pi_{v_0}$ formed using  nonzero vectors $w_{v_0} \in \pi_{v_0}$ and $w_{v_0}^{\vee}  \in \pi_{v_0}^{\vee}$ such that
\[  \int_{N(k_{v_0})} \chi(u)^{-1}  \cdot \langle w_{v_0}^{\vee}, u \cdot w_{v_0} \rangle  \, du  \ne 0. \] 
This is always possible since $\pi_{v_0}$ is supercuspidal and $(N(k_{v_0}), \chi_{v_0})$-distinguished. 
\vskip 5pt

More precisely, let $\ell \in \Hom_{N(k_{v_0})}(\pi_{v_0}, \chi_{v_0})$ be a nonzero element. For any nonzero vector $w_1 \in \pi_{v_0}$, one has
$\pi_{v_0}(C^{\infty}_c(H(k_{v_0}))) \cdot  w_1  =\pi_{v_0}$ and so there exists $\varphi \in C^{\infty}_c(H(k_{v_0})) $ such that  $\ell ( \pi_{v_0}(\varphi) \cdot w_1) \ne 0$. On the other hand, the map $\phi \mapsto \pi_{v_0}(\phi)$ is a $H(k_{v_0}) \times H(k_{v_0})$-equivariant projection 
\[  C^{\infty}_c(H(k_{v_0}))  \longrightarrow {\rm End}^{\infty}(\pi_{v_0} )  = \pi_{v_0} \otimes \pi_{v_0}^{\vee},\]
onto the maximal $\pi_{v_0}\otimes \pi_{v_0}^{\vee}$-isotypic quotient of $C^{\infty}_c(H(k_{v_0}))$.  Since $\pi_{v_0}$ is supercuspidal, this quotient in fact occurs as a submodule and 
\[  C^{\infty}_c(H(k_{v_0}))   \cong( \pi_{v_0} \otimes \pi_{v_0}^{\vee})   \oplus  \mathcal{C}'  \]
where $\mathcal{C}'$ does not contain $\pi_{v_0} \otimes \pi_{v_0}^{\vee}$ as a subquotient.  Moreover, the submodule $\pi_{v_0} \otimes \pi_{v_0}^{\vee}$ is realized by the formation of matrix coefficients of $\pi_{v_0}^{\vee}$. Hence, since $\pi_{v_0}(\varphi) \ne 0$, we may assume that  
\[  \varphi(h)  = \langle w, h \cdot w^{\vee} \rangle  \]
is a matrix coefficient of $\pi_{v_0}^{\vee}$.  Then
\begin{align}
 0 \ne \ell (\pi_{v_0}(\varphi) \cdot w_1)  & = \int_{H(k_{v_0})} \varphi(h)  \cdot \ell(h \cdot w_1)  \, dh   \notag \\
 &= \int_{N(k_{v_0}) \backslash H(k_{v_0})} \ell (h \cdot w_1)  \cdot\left(  \int_{N(k_{v_0})} \chi(u)\cdot  \varphi(uh) \, du \right)  \, dh  \notag
 \end{align}
 Thus, for some $h$, the inner integral is nonzero, as desired.

 \vskip 5pt

 Now let $C_{S_0}$ be the support of $f_{S_0} = \prod_{v \in S_0} f_{v_0}$ and define 
$C = \prod_v C_v$ as above.
\vskip 5pt

\item for $v \notin S_0$,  let $f_v$ be the characteristic function of $C_v$. 
\end{itemize}
Then $C  = \prod_v C_v$ is the support of $f$.
\vskip 5pt

Now we consider the Poincar\'e series associated to $f$:
\[  P(f) (h)  =  \sum_{\gamma \in  H(k)}  f(\gamma h), \]
so that $P(f)  \in C^{\infty}( H(k) \backslash H(\A))$. The fact that $\pi_{v_0}$ is supercuspidal implies that $P(f)$ is a cuspidal function (i.e. all its constant terms vanish).
Since $P(f)$ is smooth, it follows by \cite[Prop. 5.9]{BJ} that $P(f)$ is a cusp form and hence has compact support on $H(k) \backslash H(\A)$ by \cite[Prop. 5.2]{BJ}. In particular, $P(f)  \in L^2(H(k) \backslash H(\A))$. 
\vskip 5pt
 
To prove Theorem \ref{T:main}, 
we need to show that  $P(f)$ is globally $(N, \chi)$-distinguished. 
 We have:
\[
W_{N,\chi}(P(f))  =  \int_{N(k) \backslash N(\A)} \chi (u)^{-1}  \cdot P(f) (u)  \, du
=  \int_{N(k) \backslash N(\A)} \chi(u)^{-1}  \cdot \sum_{\gamma \in  H(k)} f(\gamma u) \, du   
\]
It suffices to sum over those $\gamma \in  H(k)$ such that 
\[  \gamma \in  H(k) \cap  C \cdot C_N. \]
 Hence, when $S_1$ and $S_2$ are sufficiently large,  Lemma \ref{L:key} implies that 
 \[  H(k) \cap C \cdot C_N  \subset N(k). \]
Thus,
\[
W_{N,\chi}(P(f))(1)  = \int_{N(k) \backslash N(\A)} \chi(u)^{-1}  \cdot \sum_{\gamma \in N(k)} f(\gamma u) \, du   =  \int_{N(\A)}   \chi(u)^{-1}  \cdot  f(u)  \, du 
=  \prod_v  W_v(f_v),  \]
where
\[  W_v(f_v)  =  \int_{N(k_v)} \chi_v(u)^{-1}  \cdot f_v(u)  \, du. \]
 Moreover, it follows by construction that for all $v$,
  \[  W_v(f_v)  \ne 0,  \]
  and for almost all $v$, one has $W_v(f_v)  = 1$.
Thus, we have shown that $W_{N,\chi}(P(f))  \ne 0$, so that $P(f)$ is  globally $(N,\chi)$-distinguished.  The spectral decomposition of $P(f)$ in $L^2(H(k) \backslash H(\A))$ then gives a cuspidal representation $\Pi$
such that $\Pi_{v_0} \cong \pi_{v_0}$ for $v_0 \in S_0$ and for all $v \notin S_0$, $\Pi_v$ has nonzero fixed vectors under a pro-$p$ Sylow subgroup of an Iwahori subgroup of $H(k_v)$.
It follows from results of Morris \cite{Mo} and Moy-Prasad \cite[Prop. 6.7 and Theorem 6.11]{MP} that for all $v \notin S_0$, $\Pi_v$ is a constituent of  a depth zero principal series representation induced from a Borel subgroup. This proves the theorem in the split semisimple case.
\vskip 10pt

\subsection{\bf General semisimple case.}
We may now consider the case when $H$ is a general semisimple group over $k$. Let $T$ be a maximal $k$-torus of $H$
containing a maximal $k$-split torus. Let $E/k$ be the splitting field of $T$, so that $H_E = H \times_k E$ is split over $E$. It is important to note that $E$ is a separable extension of $k$, since all tori over $k$ are split over a separable closure of $k$.  Choose an embedding 
\[  H_E \hookrightarrow    \GL_n (E)  \quad \text{over $E$} \]
as in the split case; it induces a $k$-embedding
\[  \iota:  H \hookrightarrow {\rm Res}_{E/k}  H_E  \hookrightarrow {\rm Res}_{E/k} \GL_n. \]
In particular, the intersection of $H_E$ with the upper triangular Borel subgroup of $\GL_n(E)$ is a Borel subgroup $T_E \cdot U_E$ of $H_E$, and $N_E =  N \times_k E  \subset U_E$.
As in the split case, we may write $U_E  = N_E  \times N_E'$ as the product of two affine subspaces. 
 Moreover, the $\mathcal{O}_{E, S_0}$-integral structure of $\GL_n(E)$ induces one on $H_E$ and an $\mathcal{O}_{S_0}$-integral structure on $H$. 
 For any $\gamma \in H(k)$, we may regard $\gamma$ as a matrix $(x_{ij}(\gamma))$ with $x_{ij}(\gamma)  \in E$. 
 
 \vskip 5pt
 
 Now let $S$ be a finite set of places of $k$ such that for all $v \notin S \cup S_0$,
 \vskip 5pt
 
 \begin{itemize}
 \item the groups $H \hookrightarrow {\rm Res}_{E/k}  H_E  \hookrightarrow {\rm Res}_{E/k} \GL_n(E)$ are smooth reductive groups over $\mathcal{O}_v$, so that their groups of $\mathcal{O}_v$-points are hyperspecial maximal compact subgroups.
 \vskip 5pt
 
 \item  $N \hookrightarrow {\rm Res}_{E/k} N_E  \hookrightarrow {\rm Res}_{E/k} U_E$ are closed immersions of smooth  unipotent group schemes over $\mathcal{O}_v$. 
 
 \vskip 5pt
 
 \item the intersection of $H(E_v)$ with the upper triangular and lower triangular Iwahori subgroups of $\GL_n(E_v)$   are Iwahori subgroups of $H(E_v)$ (where $E_v  = E \otimes_k  k_v$);  for each $w$ lying over $v$,  we denote the Iwahori subgroups of $H(E_w)$ by $I_{w}^+$ and $I_{w}^-$ respectively. 
 \vskip 5pt
 
 \item $C_{N,v} = N(\mathcal{O}_v)$ and $\chi_v$ is trivial when restricted to $N(\mathcal{O}_v)$.  
 \end{itemize}
 \vskip 5pt
 
 The existence of such a finite set $S$ follows from the same argument as in the split semisimple case considered above. 
 Now fix nonempty finite sets of places $S_1$ and $S_2$ of $k$ disjoint from $S \cup S_0$  such that $E$ splits completely at any $v \in S_1 \cup S_2$.
 We fix a test function $f = \otimes_v f_v \in C^{\infty}_c(H(\A))$ as follows:
 \vskip 5pt

 \begin{itemize}
 \item for places $v\in S_0 \cup S$, we let $f_v$ be as in the split case;
 \vskip 5pt
 
  \item for places $v \in S_1$, suppose that the places of $E$ over $v$ are $w_i$, $i  =1,...,d$.  Then one has natural isomorphisms
 \[  E \otimes_k k_v  \cong \prod_i  E_{w_i}  \cong k_v \times ...\times k_v  \quad \text{($d$ times)} \] 
 inducing an isomorphism
 \[   H_v \times_{k_v}  E_v \cong \prod_i  H_{E, w_i}   \cong H_v  \times ....\times H_v. \]
 The natural embedding 
 \[  \rho: H_v  \hookrightarrow H_v \times_{k_v}  E_v  \cong  \prod_i H_{E, w_i} \]
  is diagonal, in the sense that the projection onto any factor of the product $\prod_i  H_{E, w_i}$ is an isomorphism over $k_v$. Moreover, by our choice of $S$, 
  the preimage under $\rho$  of $\prod_i H(\mathcal{O}_{w_i})$ in $H(k_v)$ is the hyperspecial maximal compact subgroup $H(\mathcal{O}_v)$.  In particular, 
  \[ H(k_v)  \cap \rho^{-1} \left(  I_{w_1}^+  \times \prod_{i >1} H(\mathcal{O}_{w_i})  \right)  \]
  is an Iwahori subgroup $I_v^+$ of $H(k_v)$.  We then take $f_{v}$ to be the characteristic function of the pro-$p$ radical $J_v^+$ of $I_v^+$. 
 \vskip 5pt
 
 \item for places $v \in S_2$, the analogous discussion as for $S_1$ defines an Iwahori subgroup $I_v^-$ of $H(k_v)$, and we let $f_v$ be the characteristic function of $I_v^-$. 
 \vskip 5pt

 \item for all other places $v$, let $f_v$ be the characteristic function of  $K_v= H(\mathcal{O}_v)$.
 \end{itemize}
 \vskip 5pt
 
Let $C_f = \prod_v C_{f,v}$ be the support of $f$. Then we claim that Lemma \ref{L:key} continues to hold, i.e.
\[  H(k)  \cap C_f \cdot C_N  \subset N(k). \]
To see this, note that:
\vskip 5pt

\begin{itemize}
\item at places $w$ of $E$ lying above places of $k$ outside $S \cup S_0 \cup S_1 \cup S_2$, $x_{ij}(\gamma) \in \mathcal{O}_w$;

  \vskip 5pt
 \item at each place $w$ of $E$ lying above $S \cup S_0$, the maximal order of poles of $x_{ij}(\gamma)$ (for all $i,j$) is at most some integer $M$;
 \vskip 5pt
 
 \item at places $w$ of $E$ lying over $S_1 \cup S_2$, $x_{ij}(\gamma)$ lies in $\mathcal{O}_w$. Moreover, 
 for each $v_1 \in S_1$, there is at least one place $w_1$ lying over $v_1$ such that  $x_{ij}(\gamma)$ vanishes at $w_1$ for all $i > j$ and $x_{ii}(\gamma)-1$ vanishes at $w_1$ for all $i$.  Together with the above, we see that if $S_1$ is sufficiently large, $x_{ij}(\gamma)  = 0$ for all $i>j$ and $x_{ii}  =1$ for all $i$. 
\vskip 5pt

\item We have thus shown that
\[  H(k)  \cap C_f \cdot C_N  \subset U_E(E)  \subset H(E). \]
To show the desired statement, it remains to show that
\[  H(k)  \cap C_f \cdot C_N  \subset N(E).  \]
 This follows by the same argument as in the split case (if $S_2$ is sufficiently large). 
 \end{itemize} 
 \vskip 5pt
 
\noindent With the key lemma in hand,  we may now form the Poincar\'e series $P(f)$ and show the nonvanishing of  $W_{N, \chi}(P(f))$ by the same argument as in the split case. 
 Hence Theorem \ref{T:main} is proved when $H$ is semisimple.

 \vskip 10pt

 \subsection{\bf Reductive case.} \label{SS:reductive}
 We now deal with the general reductive case. Consider the semisimple group $\bar{H} := H/ Z$ over $k$ and let
 \[  r:  H \longrightarrow \bar{H}  = H/Z \]
 be the natural projection map.  
  In the semisimple case, 
    we have constructed an open compact subset of $\bar{H}(\A)$  of the form
  \[ \bar{C} =   r(C_{S_0})  \times  \prod_{v \in S} \bar{I}_v  \times \prod_{v \in S_1}\bar{J}_v^+  \times \prod_{v \in S_2}  \bar{I}_v^-   \times \prod_{\text{other $v$}} \bar{K}_v  \]
    where 
    \vskip 5pt
    
  \begin{itemize}
  \item $C_{S_0}$ is the support of an appropriate matrix coefficient of  the given supercuspidal representation $\pi_{S_0}$ of $H(k_{S_0})$, which is compact modulo $Z(k_{S_0})$;
    \item $S$ is some finite set of places and $\bar{I}_v$ is   some Iwahori subgroup $\bar{I}_v$ of $\bar{H}(k_v)$ for $v \in S$;
    \item $S_1$ is some finite set of places and $\bar{J}_v^+$ is the pro-$p$ unipotent radical of some Iwahori subgroup $I_v^+$ of $\bar{H}(k_v)$ for $v \in S_1$;
    \item  $S_2$ is some finite set of places and $\bar{I}_v^-$ is  some Iwahori subgroup of $\bar{H}(k_v)$ for $v \in S_2$; 
 \item $\bar{K}_v$ is a hyperspecial maximal compact subgroup of $\bar{H}(k_v)$ for all other $v$'s. 
  \end{itemize}
  The key property of $\bar{C}$ is that expressed in Lemma \ref{L:key}:
  \[  \overline{H}(k)  \cap \bar{C} \cdot r(C_N)  \subset r(N(k)). \] 
  In the construction of $\bar{C}$, we may further assume that $S$ is taken to be so large that, in addition to the conditions satisfied in the construction for $\bar{H}$, the groups 
   $H$ and $Z$ and the character $\omega$ of $Z$ are unramified outside $S$.      
      
  \vskip 10pt
  Note that for any Iwahori subgroup $\bar{I}_v$ of $\bar{H}(k_v)$, there is a unique Iwahori subgroup $ I_v$ of $H(k_v)$ such that 
  \[  r^{-1}(\bar{I}_v)  = Z(k_v) \cdot I_v. \]
  Indeed, the Bruhat-Tits building $\mathcal{B}(H(k_v))$ of $H(k_v)$ projects onto the building $\mathcal{B}(\bar{H}(k_v))$ of $\bar{H}(k_v)$; if $\bar{I}_v$ is associated with a chamber $\bar{\mathcal{C}}_v \subset \mathcal{B}(\bar{H}(k_v))$, then $I_v$ is associated with the unique chamber  $\mathcal{C}_v \subset \mathcal{B}(H(k_v))$ projecting onto $\bar{\mathcal{C}}_v$.
 Likewise, there is a unique hyperspecial maximal compact subgroup $K_v$ such that  $r^{-1}(\bar{K}_v) = Z(k_v) \cdot K_v$.
  The preimage of $\bar{C}$ in $H(\A)$ is
  \[  r^{-1}(\bar{C})  = C_{S_0}  \times  \prod_{v \in S} r^{-1}(\bar{I}_v)  \times \prod_{v \in S_1} r^{-1}(\bar{J}_v^+)  \times \prod_{v \in S_2}  r^{-1}(\bar{I}_v^-)   \times \prod_{\text{other $v$}} r^{-1}(\bar{K}_v).  \]
\vskip 5pt

We shall modify the subgroup $r^{-1}(\bar{I}_v) = Z(k_v)  \cdot I_v$ at the places $v \in S$ slightly. For $v \in S$, let $J_v$ be the pro-$p$ radical of $I_v$ and consider 
\[ J_v^{der} =  H^{der}(k_v) \cap J_v, \]
 which is the pro-$p$ radical of the Iwahori subgroup $I^{der}_v =  H^{der}(k_v) \cap I_v$ of the derived group $H^{der}(k_v)$.
 The subgroup $Z(k_v) \cdot J^{der}_v$ of $H(k_v)$ is compact modulo $Z(k_v)$.  Moreover, observe that
 \[  Z(k_v)  \cap J^{der}_v  = \{1 \}. \]
 Indeed, $Z(k_v)  \cap J_v^{der}  \subset Z(k_v) \cap H^{der}(k_v)$ is a finite $p$-group. However, in $Z(k_v)$, there are no nontrivial elements of finite $p$-power order, because $k$ has characteristic $p$.  Hence
 \[  Z(k_v) \cdot J^{der}_v =  Z(k_v)  \times J^{der}_v. \]
 
\vskip 5pt 
 
  Now we set
\[  C = C_{S_0} \times  \prod_{v \in S} Z(k_v) \cdot J_v^{der} \times \prod_{v \in S_1} r^{-1}(\bar{J}_v^+)  \times \prod_{v \in S_2}  r^{-1}(\bar{I}_v^-)   \times \prod_{\text{other $v$}} r^{-1}(\bar{K}_v),  \]
and note that $C \subset r^{-1}(\bar{C})$.  We may define the following test function $f = \prod_v f_v$:
  \vskip 5pt
  \begin{itemize}
  \item for the places $v_0 \in S_0$, let $f_{v_0}$ be a matrix coefficient of $\pi_{v_0}$ as in the semisimple case;
  \vskip 5pt
  \item for all other places $v$, let $f_v$ be supported on $C_v$, equivariant under $Z(k_v)$ with respect to $\omega_v$ and equal to $1$ on the relevant  compact subgroups $J^{der}_v$, $J_v^+$, $I_v^-$ or $K_v$.  This is possible because, by construction, $\omega_v$ is trivial on the intersection of $Z(k_v)$ with the relevant  compact subgroup.
 In particular, our discussion above says that at places in $S$, there is no compatibility to check; this is the main reason for using $J^{der}_v$ for $v \in S$.
   \end{itemize}
   \vskip 5pt

 Now $f$ is supported on $C$ and equivariant with respect to $\omega$ under $Z(\A)$, so that $f$ is left $Z(k)$-invariant.
  Define the Poincar\'e series 
  \[  P(f)(h)  = \sum_{\gamma \in Z(k) \backslash H(k) } f(\gamma h). \]
     Then $P(f)$ is a cuspidal automorphic function on $H(k)\backslash H(\A)$ with central character $\omega$ under $Z(\A)$ and thus belongs to the space $L^2_{\omega}(Z(\A) H(k) \backslash H(\A))$ of functions which are $(Z(\A), \omega)$-equivariant  and square-integrable on $Z(\A) H(k)  \backslash H(\A)$. 
  Moreover,  observe that the projection $r$ induces an injection:
  \[  r:    Z(k) \backslash \left( H(k)  \cap C \cdot C_N  \right)  \hookrightarrow  \bar{H}(k)  \cap \bar{C} \cdot r(C_N) \subset r(N(k)). \]
  Hence,
  we deduce that
  \[  H(k)  \cap C \cdot C_N   \subset Z(k)  \cdot N(k). \]
  Then
  \[  W_{N,\chi}(P(f)) = \int_{N(\A)}  f(n) \cdot \chi(n)^{-1} \, dn = \prod_v W_v(f_v) \ne 0 \]
  as before.  This completes the proof of Theorem \ref{T:main}.

 \vskip 10pt

 \section{\bf Proof of Corollary \ref{T:main2}.}  \label{S:corollary}
 
In this section, we give the proof of Corollary \ref{T:main2} and will use the notations in the corollary. 
The corollary requires us to globalize a number of objects and we need to deal with each in turn before we are in a position to apply Theorem \ref{T:main}.
\vskip 5pt

 \subsection{\bf Globalizing the field.}  \label{SS:global-group}
 Consider the local field $F \cong \mathbb{F}_q ( \! ( t ) \! )$. In many applications, it suffices to simply take
 \[  k_0 = \mathbb{F}_q(t) \]
  to be the function field of $\mathbb{P}_1$ over $\mathbb{F}_q$ and a place $v_0$ of $k_0$ such that $k_{0, v_0} \cong F$. However, we shall also need the following well-known fact;  we briefly explain how it can be achieved by Krasner's lemma.
  \vskip 5pt
  
  \begin{lemma}  \label{L:krasner}
 Given a finite Galois extension $E/F$ of local fields, one can find a finite Galois extension $k_1/k$ of global fields with $[k_1: k] = [E:F]$ and  a place $v$ of $k$  such that $k_{v} \cong F$ and $k_1 \otimes_k k_{v} \cong E$. In particular, the natural map ${\rm Gal}(E/F) \hookrightarrow {\rm Gal}(k_1/k)$ is an isomorphism. 
\end{lemma}

\vskip 5pt

\begin{proof}
Suppose that $E  = F(\alpha)$ (by separability) and let $f(x) \in F[x]$ be the (irreducible) minimal polynomial of $\alpha$. Then $E$ is the splitting field of  $f$.  Let $f_0 \in k_0[x]$ be $v_0$-adically sufficiently close to $f$ coefficient-wise, so that $f_0$ is also irreducible over $F$. By Krasner's lemma, there is a root $\alpha_0$ of $f_0$ which is close to $\alpha$ such that $E = F(\alpha_0)$.   Thus, the global field $k_0(\alpha_0)$ satisfies 
 \[  k_0(\alpha_0) \otimes_{k_0} F \cong  k_0(\alpha_0) \cdot F  =  E. \] 
  However, the extension $k_0(\alpha_0) / k_0$ may not be Galois.
  
  \vskip 5pt
  
  Since $E/F$ is Galois, $E$ contains all the roots of $f_0$.   Let $k_1$ be the Galois closure (in $E$) of $k_0(\alpha_0)$ over $k_0$. Then for any place $v_1$ of $k_1$ lying over $v_0$, we have 
   \[  k_{1, v_1} =  k_1 \cdot F = E, \]
    so that
  the associated decomposition group at $v_1$ is isomorphic to ${\rm Gal}(E/F)$. Thus, if we let $k$ be the fixed field of this decomposition group, we obtain an extension $k_1/k$ with $v$ the unique place of $k$ lying below $v_1$ and $v$ is inert in $k_1$. Then we have    $k_v \cong F$ and $k_1 \otimes_k k_{v} \cong E$, so that ${\rm Gal}(E/F) \cong {\rm Gal}(k_1/k)$.
 \end{proof} 
    \vskip 5pt
  
  \subsection{\bf Globalizing the groups.}
  Next we consider the question of globalizing the pair $P_F \subset H_F$. We have:
  \vskip 5pt
  
  \begin{lemma}  
  Given a parabolic $F$-subgroup $P_F \subset H_F$, one can find:
  \vskip 5pt
  
  \begin{itemize}
  \item   a global function field $k$ with a place $v_0$ such that $k_{v_0} \cong F$;
  \vskip 5pt
  
 \item a pair $P \subset H$ over $k$ such that $P$ is a parabolic $k$-subgroup of a connected reductive group $H$ over $k$ with $H_{v_0} \cong H_F$ and $P_{v_0} \cong P_F$. 
  \end{itemize} 

\noindent  Moreover, if $Z$ is the identity component of the center of $H$, so that $Z_{v_0} \cong Z_F$, then one can ensure that the $k$-rank of $Z$ is equal to the $F$-rank of $Z_F$.
     \end{lemma}
  \vskip 5pt
  
  \begin{proof}
  Assume first that $H_F$ is a quasi-split group.
  Let $H_s$ be the split form of $H_F$ which is a Chevalley group defined over $\Z$. 
Fix a Borel subgroup $B_s$ of $H_s$ containing a maximal split torus $T_s$.    Then $H_s$ determines a based root datum 
\[  \Psi = \Psi(H_s, T_s, B_s) =  (X(T_s), \Delta(T_s, B_s), Y(T_s), \Delta(T_s, B_s)^{\vee})   \]
where $X(T_s)$ and $Y(T_s)$ denote the character and cocharacter groups of $T_s$ respectively, whereas $\Delta(T_s, B_s)$ and $\Delta(T_s, B_s)^{\vee}$ denote the set of simple roots and simple corrupts respectively.
 The outer automorphism  group ${\rm Out}(H_s)$ of $H_s$ is a constant group scheme (defined over $\Z$) which is naturally isomorphic to ${\rm Aut}(\Psi)$.   
\vskip 5pt

Now $H_F$ corresponds to an element in the Galois cohomology set $H^1(F, {\rm Out}(H_s))$.  In fact, if $H_s$ is split by the finite Galois extension $E$ of $F$, $H_F$ determines an element in 
$H^1( {\rm Gal}(E/F),   {\rm Out}(H_s))$.  We pick a 1-cocycle $c: {\rm Gal}(E/F) \longrightarrow {\rm Out}(H_s)$  representing the element $H_F$; it is simply a group homomorphism. This induces an action of ${\rm Gal}(E/F)$ on the based root datum $\Psi$. Much of the structure of $H_F$ is controlled by the Galois module $\Psi$. For example, the $F$-rank of $Z_F$ is the dimension of the ${\rm Gal}(E/F)$-fixed space in $Y(Z_F) \otimes_{\Z} \Q \subset Y(T_F) \otimes_{\Z} \Q$. Moreover, conjugacy classes of parabolic $F$-subgroups of $H_F$ are in bijection with subsets of  ${\rm Gal}(E/F)$-orbits on $\Delta(T_s, B_s)$. 
\vskip 5pt

By  Lemma \ref{L:krasner}, we can find a finite Galois extension $k_1$ of $k$ and a place $v_0$ of $k$ such that
$k_1 \otimes_k k_{v_0}  = k_1 \otimes_k F  \cong E$ and  ${\rm Gal}(k_1/k)$ is naturally isomorphic to ${\rm Gal}(E/F)$.  By composition with this isomorphism, $c$ gives rise to a 1-cycle ${\rm Gal}(k_1/k) \longrightarrow {\rm Out}(H_s)$. 
This in turn gives rise to a quasi-split group $H$ over $k$, containing a pair $T \subset B$ of maximal torus contained in a Borel $k$-subgroup, which  globalizes  $H_F$, $T_F$ and $B_F$. 
 Moreover, the corresponding action of ${\rm Gal}(k_1/k)$ on the based root datum $\Psi$ is the same as that of ${\rm Gal}(E/F)$ (under the isomorphism of the two Galois groups).
Thus, the $k$-rank of $Z$ is the same as the $F$-rank of $Z_F$.  Finally, since parabolic subgroups of a quasi-split group are in bijection with subsets of Galois orbits on the set of simple roots in $\Psi$, there is a parabolic subgroup $P$ of $H$ 
whose localisation at $v_0$ is $P_F$.  This proves the lemma for $H_F$ a quasi-split group. 
\vskip 10pt

Now suppose that $H_F$ is an inner form of a quasi-split  group $H_F'$, so that $H_F$ gives rise to an ${\rm Aut}(H_F')$-orbit  in $H^1(F, H'_{F, ad})$, where $H'_{F, ad}$ is the adjoint group of $H_F'$. 
The quasi-split group $H'_F$ will contain a parabolic subgroup $P'_F$ which is a form of $P_F$. By what we showed above, we can 
find a global field $k$ with a place $v_0$ with $k_{v_0} \cong F$ such that we may  globalize the pair $P_F' \subset H_F'$ to $P' \subset H'$ as in the lemma.
\vskip 5pt

It is known that the natural map
\[  H^1(k,  H'_{ad})  \longrightarrow H^1( F, H'_{F,ad})  \]
is surjective. In characteristic $0$, this is a result of Borel-Harder \cite[Theorem 1.7]{BH}, which has been extended to positive characteristic by N. Q. Th\v{a}\'ng and N.D. T\^an \cite[Theorem 3.8.1]{TT}.
This shows that one can globalize $H_F$ to a $k$-group $H$. However, we need to be more careful if we want to globalize the parabolic subgroup $P_F$ as well.  
\vskip 5pt

For this, let ${\rm Inn}(H'_F, P'_F)$ denote the inner automorphism group of the pair $P'_F \subset H'_F$. Since parabolic subgroups are self-normalizing,  
  ${\rm Inn}(H'_F, P'_F)  =  P'_{F,ad}$ (the image of $P'_F$ in $H'_{F,ad}$). Over $k$, one similarly has ${\rm Inn}(H', P') = P'_{ad}$.
  Then we need to show that the map
  \[  H^1(k, P'_{ad})  \longrightarrow H^1(F,  P'_{F, ad}) \]
  is surjective.  
  But if $M'_{F,ad}$ is the Levi factor of $P'_{F,ad}$, then 
  \[  H^1(F, P'_{F, ad}) \cong  H^1(F, M'_{F, ad})  \]
  and likewise over $k$.  
  Hence we need to show the surjectivity of
   \[  H^1(k, M'_{ad})  \longrightarrow H^1(F,  M'_{F, ad}). \]

  \vskip 5pt

  Let $M'_{F, ad, der}$ be the derived group of $M'_{F, ad}$ so that  $M'_{F, ad, der}$ is semisimple and 
  $A_F :=  M'_{F, ad}/ M'_{F, ad, der}$ is a split torus. One has the analogous objects over $k$. Then the long exact sequence in Galois cohomology gives rise to a commutative diagram 
\[  \begin{CD}
H^1(k, M'_{ad,der})  @>>>  H^1(k, M'_{ad}) @>>> H^1(k, A) =  0  \\
@VVV   @VVV     \\
H^1(F, M'_{F,ad,der})  @>>>  H^1(F, M'_{F, ad}) @>>>H^1(F, A_F) =   0. \end{CD} \]
The first vertical arrow is surjective by the result of Th\v{a}\'ng-T\^an  \cite[Theorem 3.8.1]{TT} alluded to above. It follows that the second vertical arrow is also surjective, so that $P_F \subset H_F$ can be globalised to $P \subset H$.  
Moreover, since inner automorphisms act as identity on the center of $H_F'$ or $H'$, it is clear that the $F$-rank of $Z_F$ is the same as the $k$-rank of $Z$.
This proves the lemma.
     \end{proof}
  
  \vskip 10pt
  
  \subsection{\bf Globalizing the character $\chi_F$.}
  By the above lemma, we now have a pair $P = M \cdot N \subset H_0$ over $k$ globalizing $P_F \subset H_F$ over $F$, with $N$ the unipotent radical of $P$. Set $W = \Hom(N, \G_a)$ which is a vector group. 
  If we fix a nontrivial character $\psi:  k \backslash \A \longrightarrow \C^{\times}$, 
  then composition with $\psi = \prod_v \psi_v$ gives an identification 
  \[   W_{k} =  \Hom_k(N,  \mathbb{G}_a)  \cong  \{ \text{unitary characters of $N(k) \backslash N(\A)$} \}. \]
  Similarly, composition with $\psi_{v_0}$ gives an identification
  \[ W_F   = \Hom_F(N_F,  \mathbb{G}_a)   \cong  \{ \text{unitary characters of $N_F$} \}\]
  so that extracting the $v_0$-component of an automorphic unitary character of $N$  corresponds to the natural inclusion $W_{k} \subset W_F$. Since the $M_F$-orbit of $\chi_F$ is open (in the Zariski topology of $V$ 
  and hence in the $v_0$-adic topology of $V_F$), and $W_{k}$ is dense in $W_F$, the $M_F$-orbit of $\chi_F$ contains an element of $W_{k}$. Thus, there is an automorphic unitary character $\chi$ of $N$ whose local component at $v_0$ 
  is in the same $M_F$-orbit as $\chi_F$. 
  \vskip 5pt

 \subsection{\bf Globalizing central character.}
 Finally, we need to globalize the central character $\omega_F$. Recall that we have globalised $H_F$ to $H$ over $k$ so that the $k$-rank of the connected center $Z$ of $H$  is the same as the $F$-rank of $Z_F$. 
 
  \vskip 5pt

 \begin{lemma}  \label{L:char}
There exists an automorphic character $\omega$ of $Z$ satisfying:
  \vskip 5pt
  \begin{itemize}
  \item $\omega_{ v_0} = \omega_F$;
  \item $\omega$ is trivial on the compact group 
  \[  \prod_{v \in T} Z(k_{v})^1  \times \prod_{v \notin T \cup \{v_0\}}  Z(k_{v})^0,\]
  \end{itemize}
  where $T$ is some nonempty finite set of places of $k$, $Z(k_{v})^0$ is the maximal compact subgroup of $Z(k_{v})$ and $Z(k_{ v})^1$ is its pro-$p$ radical.
\end{lemma}
  
  \vskip 5pt
  \begin{proof}
  The proof is an elaboration of  that of \cite[Lemma 3]{P}.   To construct $\omega_0$, consider the natural map
  \[    i:  \prod_{v \in T} Z(k_{ v})^1  \times \prod_{v \notin T}  Z(k_{v})^0 \longrightarrow Z(k) \backslash Z(\A). \]
  The kernel ${\rm Ker}(i)$ is a finite group and  we shall show that it is trivial.
  Choose a splitting field $E$ of $Z$ and regard
  \[  Z(k)  \hookrightarrow Z(E)  \cong (E^{\times})^r, \]
  so that each element $z \in Z(k)$ is determined by $r$ coordinates $z_j \in E^{\times}$. If $z \in {\rm Ker}(i)$, then $z$ lies in the maximal compact subgroup of $Z(k_{v})$ for all $v$ and hence lies in $(\mathcal{O}_{E,w}^{\times})^r$ for all places $w$ of $E$. The coordinates $z_j$ of $z$ are thus constant functions on the smooth projective curve $\tilde{Y}$ with function field $E$. However, at places $w$ lying over $v \in T$, $z_j \in  1+ \varpi_w \mathcal{O}_{E,w}$, so that $z_ j$ takes value $1$ at such $w$. This implies that $z_j  =1$, so that ${\rm Ker}(i)$ is trivial,  as desired. 
  \vskip 5pt
  
  Since $i$ is injective and its image is compact and hence a closed subgroup of $Z(k) \backslash Z(\A)$, we can find a character $\omega'$ of 
   $Z(k) \backslash Z(\A)$ whose restriction to the image of $i$ is  
   \[  \omega_F|_{Z(k_{v_0})^0}  \otimes  \text{(the trivial character of  $\prod_{v \in T} Z(k_{ v})^1  \times \prod_{v \notin T \cup \{v_0 \}}  Z(k_{v})^0$). } \]
  \vskip 5pt
  
  Let $q: Z \longrightarrow \mathbb{G}_m^m$ be a surjective morphism of algebraic tori over $k$ whose kernel is anisotropic over $v_0$ (possible since the split $k$-rank of $Z$ is the same as its split $F$-rank).  Then $\omega_{v_0}' / \omega_F$  factors through
  \[  Z(k_{v_0}) \longrightarrow (k_{v_0}^{\times})^m  \longrightarrow \Z^m. \]
 Twisting $\omega_0'$ by the pullback to $Z(\A)$ of a character of the form $| -|_{\A}^{s_1} \times \cdots \times |-|_{\A}^{s_m}$, we find a character $\omega$ of $Z(k) \backslash Z(\A)$ satisfying $\omega_{ v_0} = \omega_F$, as desired.     
          \end{proof}

 \vskip 5pt
 
  \noindent{\bf Remark:}  (i)  It is necessary to know that  the split $k$-rank of the torus $Z$ is the same as the split $F$-rank of $Z_F$ above.
Consider the case when $T = \U(1)$ is an anisotropic torus of dimension $1$ over $k$ and suppose that $v_0$ is a place when $T$ splits so that $T(k_{v_0})  \cong k_{v_0}^{\times} \cong \mathcal{O}_{v_0}^{\times} \times \Z$. The irreducible representations of $T(k_{v})$ are classified by a discrete set of parameters (giving a character of the compact group $\mathcal{O}_v^{\times}$) and a continuous one (giving the image of $1 \in \Z$).  On the other hand, since $T(k) \backslash T(\A)$ is compact, its characters are classified by a discrete set of parameters. There are simply too many degrees of freedom at the place $v_0$ for every character of $T(k_{v_0})$ to be globalizable to a character of $T(k) \backslash T(\A)$. 
 \vskip 10pt
 
 (ii) In the proof of Lemma \ref{L:char}, instead of insisting that $\omega$ is trivial on $\prod_{v \in T} Z(k_v)^1$, we could have stipulated that $\omega$ restricts to any given character of 
 $\prod_{v \in T} Z(k_v)^1$. For example, one may require $\omega$ to be highly ramified at places in $T$. Then the proof of Lemma \ref{L:char} shows that  one can globalize $\omega_F$ to 
 an automorphic character which is highly ramified at places in $T$ but unramified outside $\{v_0 \} \cup T$. 
 
 \vskip 10pt
 
\subsection{\bf Proof of the corollary.}
We are now ready to complete the proof of  Corollary \ref{T:main2}. 
Let $k'$ be a finite Galois extension of $k$ which splits completely at $v_0$; suppose that  $v_0$ splits into $a$ different places $w_1, \ldots, w_a$ of $k'$.  
We may then base change the data $ (H, Z,  N, \chi, \omega)$ to $k'$.  This puts us in a position to apply Theorem \ref{T:main} and the proof of Corollary \ref{T:main2} is complete.

 \vskip 10pt

 \section{\bf Proof of Theorem \ref{T:main3}}  

In this section, we give the proof of Theorem \ref{T:main3}. The proof is a nontrivial refinement of that of \cite[Theorem 4.1]{PSP}, but instead of appealing to the relative
trace formula \cite[Theorem  4.5]{PSP} as a blackbox, we simply use the Poincar\'e series argument in the proof of Theorem \ref{T:main}. Indeed, the relative trace formula argument is simply a Poincar\'e series argument, and our treatment makes the argument in \cite{PSP} somewhat more transparent.
\vskip 5pt

By technical assumption (a), one can find a semisimple algebraic representation (over $k$)
\[  \iota:  H \longrightarrow  \GL(V)  \]
such that $R$ is the stabilizer of a vector $x_0 \in V(k)$ and ${\rm Lie}(R)$ is the infinitesimal stabilizer of $x_0$.  
Let $\mathfrak{X} \cong H/R$ be the $H$-orbit of $x_0$, so that $\mathfrak{X}$ is a locally closed subvariety of $V$ \cite[Prop. 6.7 and Theorem 6.8]{B}.  
Let $E$ be a splitting field of $H$, so that $\iota$ induces
\[  \iota_E:  H_E \longrightarrow \GL(V_E). \]
Now we note the following lemma, which is the only place where the semisimplicity of $\iota$ is used.
\vskip 5pt

\begin{lemma}
There is an $E$-basis of $V_E$ consisting of vectors, each of which is fixed by some maximal unipotent subgroups of $H_E$. 
\end{lemma}
\vskip 5pt

\begin{proof}
One easily reduces to the case when $V_E$ is irreducible. Let $v \in V_E$ be a highest weight vector with respect to a maximal unipotent subgroup $U_E \subset H_E$.  
Since the set $\{ h \cdot v: h \in H(E) \}$ is a spanning set of $V$, a subset of it is a basis, and the vector $h \cdot v$ is fixed by the maximal unipotent subgroup $h \cdot U_E \cdot h^{-1}$.
\end{proof}

\vskip 5pt

We fix an $E$-basis $\mathcal{B}^* = \{ e^*_1,...,e^*_n\}$ of $V^*_E$ as in the lemma, and let $\mathcal{B} = \{ e_1,...,e_n\}$ be the dual basis for $V_E$. 
The basis $\mathcal{B}$   gives an isomorphism $\GL(V_E) \cong \GL_n(E)$, and defines an $\mathcal{O}_w$-structure for $\GL(V_E)$ and $H_E$ for each place $w$ of $E$, as well as an $\mathcal{O}_v$-structure for $H$ for each place $v$ of $k$.
 As in \S \ref{SS:split-ss} (using \cite[\S 3.9]{T}),  let $S$ be a finite set of places of $k$ such that for all $v \notin S \cup S_0$,
\vskip 5pt
\begin{itemize}
\item $x_0 \in V(\mathcal{O}_w)$ for all $w$ lying over $v$;
\vskip 5pt

\item there is a maximal unipotent subgroup $U_{e_i}$ fixing each $e_i$, and is smooth over $\mathcal{O}_w$ for all $w$ lying over $v$;

\vskip 5pt
\item the natural map
\[  H \hookrightarrow {\rm Res}_{E/k} H_E \longrightarrow \GL(V_E) \]
is a map of smooth reductive group schemes over $\mathcal{O}_v$. 
\vskip 5pt

\item  the representation $\iota$ induces a rational representation $\iota_v$ over the residue field $\kappa_v$. 
\vskip 5pt

\item  $R$ is smooth over $\mathcal{O}_v$ and $R \hookrightarrow H$ is defined over $\mathcal{O}_v$;
\vskip 5pt

\item $\chi_v$ is trivial on $R(\mathcal{O}_v)$.
\end{itemize}
\vskip 5pt

For each $e \in \mathcal{B}$, fix  a maximal unipotent subgroup $U_e$ of $H_E$ fixing $e^*$. 
We choose a finite set $S_e$ of places $v $ of $k$ which split completely in $E$, such that $S_e$  is disjoint  from $S \cup S_0$. We also ensure that the $S_e$'s are pairwise disjoint as $e$ ranges over $\mathcal{B}$.  For $v \in S_e$, one has a commutative diagram
\[  \begin{CD}
H(\mathcal{O}_v)  @>>>  \prod_{w |v} H_E(\mathcal{O}_w) \cong  H(\mathcal{O}_v)^r  \\
@VVV   @VVV   \\
H(\kappa_v)  @>>>  \prod_{w|v} H_E(\kappa_w)   \cong H(\kappa_v)^r  
\end{CD} \] 
where $\kappa_v$ and $\kappa_w$ denote the residue field at the places $v$ and $w$ respectively. 
 The preimage of $U_e(\kappa_{w_1}) \times \prod_{w \ne w_1} H(\kappa_w)$ in $H(\mathcal{O}_v)$ is thus the pro-$p$ radical $J_v$  of  an Iwahori subgroup of $H(k_v)$.

\vskip 5pt
 
 We will now define a test function $f = \prod_v f_v \in C^{\infty}_c(H(\A), \chi|_Z)$ as follows:
\vskip 5pt

\begin{itemize}
\item for $v \in S_0$, let $f_v$ be a matrix coefficient of $\pi_v$ such that 
\[  \int_{Z(k_v) \backslash R(k_v)} \chi_v(r)^{-1} \cdot f_v(r)  \, dr   \ne 0. \]
\vskip 5pt

\item for $v\in S$, we shall make use of the technical assumption in Theorem \ref{T:main3} and choose an Iwahori subgroup $I_v^{der}$ of $H^{der}(k_v)$ with pro-$p$ radical $J^{der}_v$ such that $\chi_v$ is trivial on $R(k_v) \cap J_v^{der}$. Then we let $f_v$ be supported on $Z(k_v) \cdot J_v = Z(k_v)  \times J^{der}_v$ and trivial on $J^{der}_v$.
\vskip 5pt

\item for each $e \in \mathcal{B}$ and $v \in S_e$, let $f_v$
be supported on $Z(k_v)  \cdot J_v$ and  equal to $1$ on $J_v$.
\vskip 5pt

\item for all other $v$, let $f_v$ be supported on $Z(k_v)  \cdot H(\mathcal{O}_v)$ and  equal to $1$ on $H(\mathcal{O}_v)$. 
\end{itemize}

 \vskip 10pt
 
 Form the Poincar\'e series
 \[  P(f)(h)  = \sum_{\gamma \in Z(k) \backslash H(k)} f(\gamma h). \]
 Computing its $(R, \chi)$-period, one has
 \begin{align}
  W_{R,\chi}(P(f)) &= \int_{Z(\A) R(k) \backslash R(\A)} P(f) (r) \cdot \chi(r)^{-1}  \, dr   \notag \\
  &=   \sum_{\gamma \in H(k)/ R(k)}   \, \int_{Z(\A) R(k) \backslash R(\A)}  \,  \sum_{\delta \in R(k)} f( \gamma \cdot \delta \cdot r) \cdot \chi(r)^{-1}  \, dr \notag \\
  &= \sum_{\gamma \in H(k)/ R(k)}  \int_{Z(\A) \backslash R(\A)} f(\gamma r)\cdot \chi(r)^{-1}  \, dr. \notag 
  \end{align}
  Let us set
  \[  \phi_f(h)   =  \int_{Z(\A) \backslash R(\A)}   f(hr) \cdot \chi(r)^{-1}  \, dr,  \]
  so that $\phi_f$ is a compactly supported function on $H(\A) /  R(\A)  =  H(\A) \cdot x_0 \subset \mathfrak{X}(\A)$. Moreover, one has
  \[  \phi_f (x_0)   \ne 0. \]  
   
  \vskip 5pt
  
Now 
\[  W_{R,\chi}(P(f)) = \sum_{x \in H(k) \cdot x_0} \phi_f(x). \]
 In this sum, it suffices to consider 
 \[  x \in H(k) \cdot x_0   \cap {\rm supp}(\phi_f) \subset \mathfrak{X}(k). \]
\vskip 5pt

Now we have the key lemma:
\vskip 5pt
\begin{lemma}
When the $S_e$'s are sufficiently large (for all $e \in \mathcal{B}$),
   \[   H(k) \cdot x_0  \cap {\rm supp}(\phi_f)  = \{ x_0 \}. \]
  \end{lemma}
 \vskip 5pt
 
 \begin{proof}
 Suppose that $x \in V(k)$ lies in the intersection.  To show that $x = x_0$, it suffices to show: for each $e \in \mathcal{B}$, 
 \[ \alpha_e:=  \langle e^*, x-x_0 \rangle = 0  \in E. \]
 There is an integer $M$ such that 
 for all $v \in S_0 \cup S$, and all $w$ lying over $v$, the elements $\alpha_e$ have order of poles at most $M$ at each $w$. 
 On the other hand, for places $w$ lying over  $v \notin S \cup S_0$,  $\alpha_e \in \mathcal{O}_w$ and thus has no poles at $w$. 
 For places $v \in S_e$, however, observe that
 \[    x  =  u \cdot x_0  \mod \varpi_{w_1}    \quad \text{for some $u \in U_e(\kappa_{w_1})$.} \]
 Hence
 \begin{align}
   \alpha_e &=  \langle e^*,  u\cdot x_0  - x_0 \rangle   \mod \varpi_{w_1}   \notag \\
   &= \langle u^{-1} \cdot e^* - e^*, x_0 \rangle \mod \varpi_{w_1}  \notag \\
   &=  0  \mod \varpi_{w_1}  \notag
   \end{align}
 since $u^{-1} \cdot e^* = e^*$. Thus, we see that for each $v \in s_e$, $\alpha_e$ vanishes at some $w$ lying over $v$. It is now clear that if $S_e$ is sufficiently large, one must have $\alpha_e  = 0 \in E$. This proves the lemma.
  \end{proof}
 \vskip 5pt
 
 By the lemma, we deduce that
 \[   W_{R,\chi}(P(f))   =  \sum_{x \in H(k)\cdot x_0} \phi_f(x)  = \phi_f(x_0)  \ne 0. \]
 Moreover, $P(f)$ is fixed (at least) by the pro-$p$ unipotent radical of some Iwahori subgroup at each $v \notin S_0$. Thus, in considering the spectral decomposition of $P(f)$, we obtain a cuspidal representation $\Pi$ which globalizes the given $\pi_i$'s at $S_0$, is globally $(R, \chi)$-distinguished and whose local components $\Pi_v$ at $v \notin S_0$ belong to principal series representations induced from minimal parabolic subgroups and which has depth $0$ when restricted to $H_v^{der}$.
  \vskip 5pt
 
 This completes the proof of  Theorem \ref{T:main3}. 
 
 \vskip 10pt
 \section{\bf Stability of LS Gamma Factors}  \label{S:applications}

 In this and subsequent sections, we give some applications of the globalization theorem by combining it with the recent spectacular work of V. Lafforgue \cite{La}.  We begin by summarising the results of V. Lafforgue that we need.

\vskip 5pt

 \subsection{\bf Results of V. Lafforgue.}

 For a prime number $l \ne p$,  fix henceforth an isomorphism
 \[  \iota_l:  \overline{\Q}_l  \longrightarrow \C \]
 which allows one to compare $\overline{\Q}_l$-valued functions with $\C$-valued ones.
 Given an algebraic group $H$ over a global function field $k$, the isomorphism $\iota_l$ also induces an isomorphism
 \[  \iota_l: {^L} H(\overline{\Q}_l) \cong {^L} H(\C). \]

 \vskip 5pt
 
 Now let $\Pi$ be a cuspidal automorphic representation of $H(\A)$.  
 By V. Lafforgue \cite{La},  one can associate to $\Pi$ a continuous global Galois representation
 \[  \rho = \rho_{\Pi, l} :  {\rm Gal}(k^{{\rm sep}}/k) \longrightarrow {^L} H(\overline{\Q}_l).  \]
  We list some of its properties which will be relevant:
 \vskip 5pt
 
 \begin{itemize}
 \item[(a)]  for almost all places $v$ of $k$, the (Frobenius-semisimplification of the) local Galois representation $\rho_v$ is unramified and the image of the geometric Frobenius element
 ${\rm Frob}_v$ is equal to the Satake parameter of $\Pi_v$ (after composing with $\iota_l$).    
 \vskip 5pt
 
 \item[(b)]  the embedding $Z(H)^0 \hookrightarrow  H$ induces a morphism 
 \[ \rho_Z :  {^L} H  \longrightarrow  {^L} H /  {^L}H^0_{der} \cong {^L}Z(H)^0,   \]
 and the central character $\omega_{\Pi}$ of $\Pi$ corresponds to the map $\rho_Z \circ \rho_{\Pi}$ under the global Langlands correspondence for tori (again, after composing with $\iota_l$).
 Indeed, by (a), one knows that the two characters correspond at almost all places, and so correspond by Tchebotarev's density theorem.  
 \end{itemize}
 \vskip 5pt

For each place $v$ of $k$, one obtains a continuous (Frobenius-semisimplified) $l$-adic local Galois representation
\[  \rho_{\Pi, l,v}  :  {\rm Gal}( k_v^{{\rm sep}}/k_v) \longrightarrow {^L}H(\overline{\Q}_l). \]
By a well-known construction due to Grothendieck for $\GL(n)$  (see \cite{R}) and \cite[\S 2.1]{GR} in general, such an $l$-adic local Galois representation corresponds to a (Frobenius-semisimple) representation of the Weil-Deligne group 
\[    W_{k_v}  \times \SL_2(\overline{\Q}_l) \longrightarrow {^L}H(\overline{\Q}_l). \]
Using the isomorphism $\iota_l$, this gives a  (Frobenius-semisimple) Weil-Deligne representation    
\[  \rho_{\Pi,v}:  W\!D_{k_v} = W_{k_v} \times \SL_2(\C) \longrightarrow {^L}H(\C). \]
In the rest of this paper, we shall pass between the (Frobenius-semisimplified) $l$-adic representation $\rho_{\Pi, l, v}$ and the local L-parameter $\rho_{\Pi,v}$ without further comment.  We will also drop the adjective ``Frobenius semisimple" henceforth, as all our local representations or parameters will be assumed to be Frobenius-semisimplified. Note that Frobenius-semisimplification does not change local Artin L-factors or gamma factors. 
\vskip 10pt

\subsection{\bf Langlands-Shahidi gamma factor.}  \label{SS:LS}
We can now introduce the Langlands-Shahidi (LS) gamma factors.
Let $F$ be a local field of characteristic $p>0$, and   let $P_F = H_F \cdot N_F$ be a maximal parabolic $F$-subgroup of a connected reductive quasi-split group $G_F$, with Levi factor $H_F$ and unipotent radical $N_F$.  One has a natural inclusion of Langlands L-groups
\[  {^L}H_F \longrightarrow {^L}G_F. \]
Suppose that the adjoint action of ${^L}H_F$ on ${\rm Lie}(N_F)$ decomposes as:
\[  {\rm Lie}(N_F)  = \bigoplus_i r_i  \]
for irreducible representations $r_i$, $1 \leq i \leq m_r$, ordered as in \cite{Sh2} according to nilpotency class.   The second author has extended the  Langlands-Shahidi theory to the case of function fields \cite{L1,L2}. In particular, given an irreducible generic representation $\pi$ of $H_F$, one can attach a local gamma factor 
\[  \gamma(\pi, r_i, \psi)  \]
for each $r_i$, where $\psi$ is a nontrivial additive character of $F$. 
These LS gamma factors are $\C$-valued meromorphic functions in $\pi$ (as we shall explain shortly) and  satisfy some natural properties which characterize them uniquely \cite{L2}.
\vskip 5pt

Since $H_F$ is a maximal Levi subgroup of $G_F$, the quotient $H_F/Z(G_F)^0$ of $H_F$ by the connected center $Z(G_F)^0$ of $G_F$ has 1-dimensional split center, so that 
$\Hom_F(H_F /Z(G_F)^0, \G_m)$ is a free $\Z$-module of rank $1$. 
Let $\delta \in \Hom_F(H_F/Z(G_F)^0, \mathbb{G}_m)$ be the generator such that  the modulus character $\det ({\rm Ad}_{H_F} | {\rm Lie }(N_F))$ is a positive multiple of $\delta$.  
 For any character $\chi:  F^{\times} \longrightarrow \C^{\times}$, the composite $\chi \circ \delta$ is a 1-dimensional character of $H_F$ which is trivial on   $Z(G_F)$. In particular, for an irreducible representation $\pi$ of $H_F$, one may consider the twist $\pi \otimes (\chi \circ \delta)$. We shall denote this twisted representation simply as $\pi \otimes \chi$
so that we have $\gamma( \pi \otimes \chi, r_i, \psi)$.  The character $\delta$ corresponds to a morphism 
\[  \phi_{\delta}: \C^{\times} \longrightarrow {^L}H_F, \]
 taking value in the center of $({^L}H_F)^0$. 
For any character $\chi:F^{\times} \rightarrow \C^{\times}$, one then has an induced  map
\[  
   \begin{CD}
   W_F @>>> F^{\times} @>\chi>> \C^{\times} @>\phi_{\delta}>>   {^L}H_F,
   \end{CD}
\]
 which is the L-parameter for the character $\chi \circ \delta$ of $H_F$. 
 For simplicity, we shall write $\chi$ for this map as well.
 
 \vskip 5pt
 
 Recall that the set of characters of $F^{\times}$ is the countable disjoint union of 1-dimensional complex manifolds. 
 As $\chi$ varies over the characters of $F^{\times}$, the function $\chi \mapsto \gamma(\pi \otimes \chi, r_i, \psi)$ is a meromorphic function. To be more precise, 
  setting  $\chi = |-|_F^s$, for $s \in \C$, then the function
  \[  \gamma(\pi \otimes |-|_F^s,  r_i,\psi)    \in \C(q_F^{-s})  \]
  is a rational function in $q_F^{-s}$.

\vskip 5pt

\subsection{\bf Stability.}  \label{SS:stability}
We now turn towards an important stability property of LS $\gamma$-factors. This is an open problem in characteristic zero, but  see \cite{CPSS} where important cases are established. The proof here in positive characteristic extends the cases of symmetric and exterior square $\gamma$-factors studied in \cite{HL1}.
\vskip 5pt

\begin{thm}  \label{T:stability}
Let $R = r_i$ for some $i$ (in the notation of \S \ref{SS:LS}).  
Let $\pi_1$ and $\pi_2$ be two irreducible generic representations of $H_F$ with the same central character. For all sufficiently highly ramified characters $\chi$ of $F^{\times}$, 
\[  \gamma(\pi_1 \otimes \chi, R, \psi) =  \gamma(\pi_2 \otimes \chi, R, \psi) \]
 
\end{thm}
\vskip 5pt

Here, note that  the set of characters $\chi$ with a fixed conductor is a complex manifold of dimension $1$ under twisting by unramified characters, and the identity in the theorem is interpreted as an identity of meromorphic function on this complex manifold. 
\vskip5pt

\begin{proof}
  We first consider a generic supercuspidal representation $\pi$ of $H_F$. Let $\chi$ be any character of $F^{\times}$.
     By Corollary \ref{T:main2}, we may find a global function field $k$ (indeed $k = \mathbb{F}_q(t)$ will do) with  $k_{v_0}  = F$ for some place $v_0$ of $k$
     and globalize the data $(H_F \cdot N_F \subset G_F, \pi)$ to $(H \cdot N \subset G, \Pi)$ where $H \subset G$ are quasi-split and $\Pi$ is a globally generic cuspidal representation as in Corollary \ref{T:main2}. 
  \vskip 5pt
  
  Without loss of generality,    we may assume that the character $\psi$  can be globalised to a character $\Psi$ of $k \backslash \A$.     Indeed, if we fix a nontrivial character $\Psi = \prod_v \Psi_v$ of $k \backslash \A$, then the nontrivial characters of $k \backslash \A$ are of the form $\Psi_a(x)  = \Psi(ax)$ for $a \in k^{\times}$, and  the set of nontrivial characters of $F$ is of the form $\Psi_{v_0, a}(x) = \Psi_{v_0}(ax)$ for $a \in F^{\times}$.  This shows that a dense subset of nontrivial characters of $F$ can be globalized to a character of $k \backslash \A$. On the other hand, for $i=1$ or $2$, one has
  \[  \gamma(\pi_i \otimes \chi, R, \psi_a) =  \alpha(a)    \cdot  \gamma(\pi_i \otimes \chi, R, \psi),  \]
   for some character $\alpha$ of $F^{\times}$ depending only on $R$, $\chi$ and  the central character of $\pi_i$.
  This shows that if the identity to be shown in the theorem holds for one nontrivial $\psi$, then it holds for all nontrivial $\psi$.
   \vskip 5pt
 
 By V. Lafforgue \cite{La}, one has a continuous  semisimple $l$-adic global Galois representation 
 \[  \rho = \rho_{\Pi}: {\rm Gal}(k^{{\rm sep}}/k) \longrightarrow {^L}H(\overline{\Q}_l) \]
 associated to $\Pi$. Let $S$ be a nonempty finite set of places of $k$ not containing the distinguished place $v_0$  such that for all $v \notin S$ different from $v_0$, 
   \vskip 5pt
   
   \begin{itemize}
   \item   $H_v$, $\Pi_v$ and $\Psi_v$ are unramified.  
  \item $\rho_v$ is unramified and $\rho_v({\rm Frob}_v)$ is the Satake parameter of $\Pi_v$. 
  \end{itemize}
  For $v \in S$, $\Pi_v$ is nonetheless contained in a  principal series representation induced from a Borel subgroup.      
  We may globalize $\chi$ to a Hecke character $\mathcal{X}$ which is unramified outside $S \cup \{v_0\}$ and highly ramified for places in $S$ (by part (ii) of the remark following Lemma \ref{L:char}). 
Then the global functional equation from Langlands-Shahidi theory gives:,
\begin{equation} \label{E:funct1} 
\prod_{v\in S \cup \{v_0\}}   \gamma( s, \Pi_v \otimes \mathcal{X}_v , R, \Psi_v)   = \frac{L^{S \cup \{ v_0\}}( s, \Pi \otimes \mathcal{X}_v, R)}{L^{S \cup \{v_0 \}}( 1-s ,\Pi \otimes \mathcal{X}_v, R^{\vee})}, \end{equation}
 where for the purpose of this proof, we have written
 \[   \gamma( s, \Pi_v \otimes \mathcal{X}_v , R, \Psi_v) :=   \gamma(\Pi_v \otimes \mathcal{X}_v |-|_v^s , R, \Psi_v)  \]
 and likewise for the L-functions on the right-hand-side. 
 \vskip 5pt

Now consider the representation 
\[  R \circ( \rho  \otimes \mathcal{X}):  {\rm Gal}(k^{{\rm sep}}/k)  \longrightarrow  \GL(V_R),\]
where we have regard $\mathcal{X}$ as a 1-dimensional representation of $ {\rm Gal}(k^{{\rm sep}}/k) $ by global class field theory.
 After composing with the isomorphism $\iota:  \overline{\Q}_l \cong \C$, one may form the global L-function of Artin type:
 \[  L(s, R \circ( \rho \otimes \mathcal{X} )) = \prod_v  L(s,  R\circ ( \rho_v \otimes \mathcal{X}_v))  \]
 which converges for $\Re(s)$ sufficiently large.  By the work of Grothendieck-et-al, one knows that this L-function is in fact a rational function in $q^{-s}$ (and hence admits  meromorphic continuation to $\C$) and satisfies a functional equation of the form
 \[  L(s, R \circ( \rho \otimes \mathcal{X}))  = \epsilon(s, R \circ (\rho \otimes \mathcal{X} ))  \cdot L(1-s, (R \circ (\rho \otimes \mathcal{X}))^{\vee} ) \]
 for some global epsilon factor $\epsilon(s, R \circ \rho \otimes \mathcal{X})$.  It is known by work of Laumon \cite[Theorem 3.1.5.4 and Theorem 3.2.1.1]{Lau1} (see also \cite{D2}),  that the epsilon factor admits an Euler product
 \[  \epsilon(s, R \circ (\rho \otimes \mathcal{X}))  = \prod_v  \epsilon(s, R \circ( \rho_v \otimes \mathcal{X}_v), \Psi_v)  \quad \text{(a finite product)} \]
 for any character $\Psi = \prod_v \Psi_v$ of $k \backslash \A$. 
 In particular, one may define the local Galois theoretic gamma factors:
 \[  \gamma(s, R \circ( \rho_v \otimes \mathcal{X}_v), \psi_v) := \epsilon(s, R \circ (\rho_v \otimes \mathcal{X}_v,) \Psi_v) \cdot \frac{L(1-s, R^{\vee} \circ (\rho_v \otimes \mathcal{X}_v))}{L(s, R\circ( \rho_v \otimes \mathcal{X}_v))}, \] 
  and the global functional equation can be expressed as
  \begin{equation}  \label{E:funct2}
      \prod_{v \in S \cup \{v_0 \}}  \gamma(s, R\circ (\rho_v \otimes \mathcal{X}_v), \Psi_v)  =   \frac{L^{S \cup \{v_0 \}}(s, R\circ (\rho_v \otimes \mathcal{X}))}{L^{S \cup \{v_0 \} }(1-s, R^{\vee} \circ( \rho_v \otimes \mathcal{X}))}. \end{equation}
\vskip 5pt

Comparing (\ref{E:funct1}) and (\ref{E:funct2}) and using the compatibility of $\rho_v$ and $\Pi_v$ outside of $S$, we deduce that
\begin{equation} \label{E:comparison}
  \prod_{v\in S \cup \{v_0\}}   \gamma(s, \Pi_v \otimes \mathcal{X}_v, R, \Psi_v)  =  \prod_{v \in  S\cup \{v_0\}}  \gamma(s, R\circ (\rho_v  \otimes  \mathcal{X}_v), \Psi_v).  \end{equation}
In particular, we have:
 \[ \frac{\gamma(s, \pi \otimes \chi, R, \psi)}{\gamma(s, R \circ (\rho_{v_0} \otimes \chi), \psi)}  = \prod_{v \in S}  \frac{\gamma(s, R\circ (\rho_v  \otimes  \mathcal{X}_v), \Psi_v)}{ \gamma(s, \Pi_v \otimes \mathcal{X}_v, R, \Psi_v)}. \]
 For $v \in S$, the representations $\Pi_v$ is a constituent of a principal series representation $I_v(\mu_v)$ induced from the Borel subgroup. As such, by multiplicativity of LS gamma factors and their compatibility with class field theory in the case of tori \cite{L2},  one has
 \[   \gamma(s, \Pi_v \otimes \mathcal{X}_v, R, \Psi_v) = \gamma(s, R \circ (\phi_v \otimes \mathcal{X}_v,) \Psi_v) \]
 where $\phi_v$ is the composite
 \[  \phi_v :  W_k \longrightarrow {^L}T_v  \longrightarrow {^L}H_v  \]
 with the first map corresponding to the inducing data $\mu_v$ on $T(k_v)$.
 In particular, we have
\[   
 \frac{ \gamma(s, \pi \otimes \chi, R, \psi)}{ \gamma(s, R \circ( \rho_{v_0} \otimes \chi), \psi)} =  \prod_{v \in S}  \frac{\gamma(s, R\circ( \rho_v  \otimes  \mathcal{X}_v), \Psi_v)}{ \gamma(s, R \circ ( \phi_v\otimes \mathcal{X}_v),  \Psi_v)}. \]
 Now we know that
\[  \rho_Z \circ \rho_v =  \rho_Z \circ  \phi_v, \]
since both these characters correspond to the central character of $\Pi_v$. Moreover, the map 
\[ \det \circ R :  {^L}H \longrightarrow \bar{\Q}^{\times}_l \cong \C^{\times}  \]
factors through $\rho_Z$.  Thus we deduce that
\[  \det R \circ \rho_v = \det R \circ \phi_v.  \]
Since $\mathcal{X}_v$ can be as highly ramified as we wish for $v \in S$, it follows by the stability of Galois-theoretic gamma factors that  for suitable choice of $\mathcal{X}_v$  for all $v \in S$, we have
\[  
 \frac{\gamma(s, R\circ (\rho_v  \otimes  \mathcal{X}_v), \Psi_v)}{ \gamma(s, R \circ ( \phi_v\otimes \mathcal{X}_v,)  \Psi_v)}  = 1. \]
 Hence, we conclude that for arbitrary character $\chi$ of $F^{\times}$, one has
\begin{equation} \label{E:gamma}
   \gamma(s, \pi \otimes \chi, R, \psi)  =  \gamma(s, R \circ( \rho_{v_0} \otimes \chi), \psi)  \end{equation}
 for a supercuspidal representation $\pi$. The same then holds for general generic $\pi$ by induction, using the multiplicativity of LS gamma factors. 
We have thus expressed arbitrary LS gamma factors in terms of (some) Galois theoretic ones. 
\vskip 10pt

We can now complete the proof of the theorem. Given two irreducible generic representations $\pi$ and $\pi'$ of $H_F$ with the same central character, we may apply the above discussion to each of them in turn. In particular, we have the equation (\ref{E:gamma}) for $\pi$ and $\pi'$, with $\rho_{v_0}$ and $\rho'_{v_0}$ satisfying
\[  \det R \circ \rho_{v_0}  = \det R \circ \rho'_{v_0}. \]
Thus  if $\chi$ is sufficiently highly ramified, 
\[  \gamma(s, R \circ( \rho_{v_0} \otimes  \chi), \psi) =    \gamma( s, R \circ (\rho'_{v_0} \otimes  \chi ), \psi).  \]
and  one obtains the desired stability of LS gamma factors.
\end{proof}
\vskip 5pt

We record here a corollary.
\vskip 5pt

\begin{cor}  \label{C:gamma-stab}
Let $\Pi = \otimes_v \Pi_v$ be a globally generic cuspidal representation of $H(\A)$ (where $H$ is a connected reductive group over the global function field $k$) and let $\rho_{\Pi}$ be the $l$-adic global Galois representation associated to $\Pi$ by V. Lafforgue. 
Let $\mathcal{X}  = \otimes_v \mathcal{X}_v$ be a Hecke character and $\Psi = \prod_v \Psi_v$ a nontrivial character of $\A / k$.  Then for all places $v$ of $k$, 
\[    \gamma( \Pi_v \otimes \mathcal{X}_v |-|_v^s , R, \Psi_v)  =  \gamma( R \circ( \rho_{\Pi, v} \otimes \mathcal{X}_v |-|_v^s), \Psi_v). \] 
\end{cor}

\begin{proof}
Fix a place $v_0$ of $k$. As in the proof of Theorem \ref{T:stability}, 
one uses the global functional equations to obtain the identity (\ref{E:comparison}), with $v_0$ the place of interest. 
In the notation of (\ref{E:comparison}), for each $v \in S$, equation (\ref{E:gamma}) gives:  
\[  \gamma(s, \Pi_v, \mathcal{X}_v, R, \Psi_v)= \gamma(s, R \circ (\phi_v \otimes \mathcal{X}_v), \Psi_v), \]
for some $\phi_v :  WD_{k_v} \longrightarrow {^L} H(\C)$ such that  $\rho_Z \circ \phi_v$ corresponds to the central character of $\Pi_v$ (under the local Langlands correspondence for the torus $Z$). Since $\rho_Z \circ \rho_{\Pi, v}$  also corresponds to the central character of $\Pi_v$, we deduce that
\[  \det R \circ \phi_v = \det R \circ \rho_{\Pi,v}.  \]
 In particular, by multiplying $\mathcal{X}$  by a character $\mathcal{X}'$ which is trivial at the place of interest $v_0$, highly ramified at all $v \in S$ and unramified outside 
 $S \cup \{ v_0 \}$, we can appeal to the stability of Galois theoretic gamma factors to deduce that
 \[   \gamma(s, R \circ (\phi_v \otimes \mathcal{X}_v\mathcal{X}'_v), \Psi_v) = 
   \gamma(s, R \circ (\rho_{\Pi, v} \otimes \mathcal{X}_v \mathcal{X}'_v), \Psi_v) \]
  for all $v \in S$. Now the desired result follows from the equation (\ref{E:comparison}) with $\mathcal{X}$ replaced by $\mathcal{X} \cdot \mathcal{X}'$.
 \end{proof}

\vskip 15pt
\section{\bf  Plancherel Measures}  \label{S:plan}
 We continue with the set up of the previous section, but now we assume that $\pi$ is any irreducible representation of $H_F$, not necessarily generic. 
In this case, one can consider the Plancherel  measure  associated to the induced representation ${\rm Ind}_{P_F}^{G_F}  \pi \otimes \chi$. More precisely, one has a standard intertwining operator \cite[\S IV]{W}
\[  M(\pi \otimes \chi, P_F, \bar{P}_F,\psi) : {\rm Ind}_{P_F}^{G_F}  \pi \otimes \chi \longrightarrow {\rm Ind}_{\bar{P}_F}^{G_F}  \pi \otimes \chi \]
 defined by the usual integral when $|\chi| = |-|_F^s$  with ${\rm Re}(s)$  sufficiently large and admits a meromorphic continuation to all $\chi$.     The composite:
 \[  M(\pi \otimes \chi, \bar{P}_F, P_F,\psi) \circ M(\pi \otimes \chi, P_F, \bar{P}_F,\psi)  = \mu( \pi\otimes \chi,\psi)^{-1} \] 
 is a scalar-valued meromorphic function (in $\chi$) known as the Plancherel measure. Indeed, the function $\mu( \pi \otimes \chi |-|_F^s, \psi)$
 is a rational function of $q_F^{-s}$. 
 \vskip 5pt
 
Since the definition of the intertwining operators depends on the choice of  Haar measures on $N_F$ and $\bar{N}_F$, $\mu(s, \pi \otimes \chi)$ is a priori only well-defined up to scaling by a positive real number. For a precise normalization, see \cite[Appendix B]{GI}.  In particular, there is a unique  normalisation such that when  $\pi$  is a constituent of a principal series representation induced from a Borel subgroup, one has
\[  \mu(\pi, \psi )  = \prod_i \gamma ( r_i \circ \phi_{\pi}, \psi) \cdot \gamma (r_i^{\vee} \circ \phi_{\pi}, \overline{\psi}), \]
with 
\[  \phi_{\pi} : W_F \longrightarrow  {^L}T  \longrightarrow {^L}G \]
 where the first map is the one attached to the inducing data for $\pi$. 

  \vskip 5pt
 
 Suppose now that $\pi$ is supercuspidal and suppose we have globalised the data $(F, P_F = H_F \cdot N_F \subset G_F,  \pi)$ to $(k, P= H \cdot N \subset G, \Pi)$ as in the previous section, with $H \subset G$ quasi-split and $\Pi$ a cuspidal representation  such that  $\Pi_{v_0} = \pi$ and $\Pi_v$ is a constituent of a principal series representation induced from a Borel subgroup for all other $v$. Let $\rho$ be the Galois representation associated to $\Pi$ by V. Lafforgue [La]. 
 The goal of this subsection is to show:
 \vskip 5pt
 
\begin{thm}  \label{T:plan}
With the above  notations, one has
 \[
  \mu( \pi \otimes |-|_F^s  , \psi )  = \prod_i  \gamma(r_i \circ  (\rho_{v_0} \otimes |-|_F^s), \psi) \cdot \gamma( r_i^{\vee}\circ (\rho_{v_0} \otimes |-|_F^s), \overline{\psi}). \]  
\end{thm}
\vskip 5pt
\vskip 5pt

\begin{proof}
 Over the global field $k$,  one has the global analog of the discussion before the statement of the theorem. When the Haar measures on $N(\A)$ and $\bar{N}(\A)$ are taken to be the Tamagawa measures, then one has the global functional equation \cite[Theorem IV.1.10]{MW}
\[  
 M(\Pi \otimes \mathcal{X}, \bar{P}, P) \circ M(\Pi \otimes \mathcal{X}, P, \bar{P})   = 1, \]
 where $\mathcal{X}$ is a global Hecke character with $\mathcal{X}_{v_0}  =1$ and $\mathcal{X}_v$ is unramified outside a finite set $S$.
 \vskip 5pt
 
Comparing this with the global functional equation of the Galois theoretic gamma factors, one sees that
\begin{equation}  \label{E:plan-com1}
 \prod_{v \in S \cup \{v_0 \}} \mu(\Pi_v \otimes \mathcal{X} |-|_v^s , \Psi_v) = \prod_{v \in S \cup \{v_0 \}} \prod_i  \gamma(r_i \circ (\rho_v \otimes \mathcal{X}_v |-|_v^s), \Psi_v) \cdot  \gamma(r_i^{\vee}\circ (\rho_v \otimes \mathcal{X}_v |-|_v^s), \overline{\Psi}_v). \end{equation}
 It follows that
 \begin{equation} \label{E:plan-com2}
 \frac{\mu(\pi  \otimes |-|_F^s, \psi)}{\prod_i  \gamma(r_i \circ (\rho_{v_0} \otimes |-|_F^s, \psi) \cdot  \gamma(r_i^{\vee}\circ (\rho_{v_0} \otimes |-|_F^s), \overline{\psi})}   \end{equation}
 \[ =  \prod_{v \in S} \prod_i   \left(
 \frac{  \gamma(r_i \circ (\rho_v \otimes \mathcal{X}_v |-|_v^s), \Psi_v) \cdot  \gamma(r_i^{\vee}\circ (\rho_v \otimes \mathcal{X}_v|-|_v^s), \overline{\Psi}_v)}{ \gamma (r_i \circ( \phi_{\Pi_v} \otimes \mathcal{X}_v |-|_v^s), \Psi_v) \cdot \gamma (r_i^{\vee} \circ( \phi_{\Pi_v} \otimes \mathcal{X}_v |-|_v^s),\overline{ \Psi}_v)} \right). \]
 Now suppose that $\mathcal{X}_v$ is sufficiently highly ramified for $v \in S$ and $\mathcal{X}_v$ is unramified for all other $v$'s.   Since $\det r_i  \circ \rho_v = \det r_i \circ \phi_{\Pi_v}$ and $\mathcal{X}_v$ is sufficiently highly ramified for $v \in S$, we conclude that
 \[     \mu(\pi \otimes |-|_F^s, \psi)  =\prod_i  \gamma( r_i \circ (\rho_{v_0} \otimes |-|_F^s), \psi) \cdot  \gamma( r_i^{\vee}\circ (\rho_{v_0} \otimes |-|_F^s), \overline{\psi}). \] 
 This completes the proof of  the theorem. 
 \end{proof}
 \vskip 10pt
 
 \begin{cor} \label{C:plan-galois}
 Let $\pi$ be any irreducible representation of $H_F$ (where $F$ be a local field). Then there is a local Galois representation $\phi:   {\rm Gal}(F^{sep}/F) \longrightarrow {^L} H_F$
 such that the (connected) central character $\omega_{\pi}$ of $\pi$ corresponds to the character $\rho_Z \circ \phi$ under the local Langlands correspondence for tori
and  
\[  
  \mu(\pi \otimes |-|_F^s, \psi )  = \prod_i  \gamma( r_i \circ  (\phi \otimes |-|_F^s), \psi) \cdot \gamma( r_i^{\vee}\circ (\phi \otimes |-|_F^s), \overline{\psi}). \]  
   \end{cor}
 
 \begin{proof}
 This follows from the supercuspidal case (as demonstrated in Theorem  \ref{T:plan}) by multiplicativity of Plancherel measures (see \cite[Appendix B]{GI}).
 \end{proof}
 
 \begin{cor}\label{plancherelstability}
 Let $\pi_1$ and $\pi_2$ be two irreducible  representations of $H_F$ with the same central character. For all sufficiently highly ramified characters $\chi$ of $F^{\times}$, 
\[   \mu( \pi_1 \otimes \chi, \psi)  =  \mu(\pi_2 \otimes \chi, \psi). \]
\end{cor}
\vskip 5pt

\begin{proof}
Since we have expressed the Plancherel measures in terms of Galois theoretic gamma factors in Corollary \ref{C:plan-galois} , the stability under twisting by highly ramified characters follows by that of the Galois theoretic ones.
\end{proof}
\vskip 5pt

\begin{cor}  \label{C:plan}
Let $\Pi$ be any cuspidal automorphic representation of $H(\A)$ and let $\rho$ be the Galois representation associated to $\Pi$ by V. Lafforgue.  For any $v$, one has
\[  \mu(\Pi_v \otimes |-|_F^s, \Psi_v)   =     \prod_i  \gamma( r_i \circ  (\rho_v \otimes |-|_v^s), \Psi_v) \cdot \gamma( r_i^{\vee}\circ (\rho_v \otimes |-|_v^s),\overline{\Psi}_v). \]  
\end{cor}

\vskip 5pt

\begin{proof}
We start with a comment on the difference between this corollary and Theorem \ref{T:plan}. In Theorem \ref{T:plan}, we started with a local supercuspidal representation and globalize it to $\Pi$ according to Theorem \ref{T:main}, so that we have control at all places outside $v_0$, which allows us to deduce the identity in the corollary  at $v_0$ (the identity being known at all other places). What the corollary asserts is that for any cuspidal $\Pi$, the same conclusion continues to hold at all places for $\Pi$. 
\vskip 5pt

The proof is via the same  argument as in the proof of Corollary \ref{C:gamma-stab}. To be more precise, fix a place $v_0$ of $k$ which is the place of interest.
The global functional equation for intertwining operators gives the equation (\ref{E:plan-com1}) in the current context. Now using Corollary \ref{C:plan-galois}, one obtains 
the analog of equation (\ref{E:plan-com2}) for some local Galois representation $\phi_v$ (in place of $\phi_{\Pi_v}$ in (\ref{E:plan-com2})). Now  one obtains the desired conclusion by appealing to the stability of Galois theoretic gamma factors as in the proof of Corollary \ref{C:gamma-stab} to isolate the place $v_0$.
\end{proof}

\vskip 10pt
\section{\bf Local Langlands Correspondence}  \label{S:LLC}
In this section, we specialize the discussion of the previous section to the case when  $H_F$ is a quasi-split classical group.
\vskip 5pt

\subsection{\bf Classical groups.}
Thus,  let $E$ be equal to $F$ or a quadratic field extension, with ${\rm Aut}(E/F)  = \langle c \rangle$ and let $V$ be a finite-dimensional vector space over $E$ equipped with a nondegenerate sesquilinear form $\langle-, -\rangle$. Then $H_F  = {\rm Aut}(V, \langle-,-\rangle)^0$, and the various possibilities are given by:
\[  H_F = \SO_{2n+1} \quad \text{or} \quad \Sp_{2n} \quad \text{or} \quad  \SO_{2n}  \quad \text{or} \quad  \U_n.  \]
  \vskip 5pt
 
\noindent The  Langlands dual group of $H_F$ is
 \[  H_F^{\vee} = \Sp_{2n}(\C) \quad \text{or} \quad \SO_{2n+1}(\C) \quad \text{or}  \quad \SO_{2n}(\C) \quad  \text{or} \quad \GL_n(\C). \]
 Following \cite{GGP}, an  L-parameter $\phi: W\!D_F \longrightarrow {^L}H_F$  for $H_F$ gives rise to an  equivalence class of self-dual or conjugate-self-dual representations
\[  \phi:  W\!D_E  \longrightarrow \begin{cases}
\Sp_{2n}(\C) \\
\SO_{2n+1}(\C) \\
\O_{2n}(\C) \\
\GL_n(\C) \end{cases} \]
of appropriate sign $\epsilon = \pm 1$ in the respective cases. Note that if $H_F = \SO_{2n}$, we are considering the L-parameters up to equivalence under conjugacy by $\O_{2n}(\C)$ and not just  by $\SO_{2n}(\C)$. To be precise, a representation $\phi$ of $W\!D_E$ is conjugate-self-dual if $\phi^c \cong \phi^{\vee}$, and the L-parameter of $H_F$ gives rise to   a conjugate-self-dual  representation of $W\!D_E$ of 
\[ 
\text{dimension}= \begin{cases}
2n \\
2n+1 \\
2n  \\
n  \end{cases}
\quad \text{and sign} = \epsilon(H_F):= 
 \begin{cases}
-1, \text{  if $H_F = SO_{2n+1}$;} \\
+1, \text{  if $H_F = Sp_{2n}$;} \\
+1 \text{  if $H_F = SO_{2n}$;} \\
(-1)^{n-1}, \text{  if $H_F =  \U_n$.}
\end{cases}  \]
We shall frequently identify an L-parameter $\phi$ of $H_F$ with its associated conjugate-self-dual representation of $W\!D_E$. 
 Likewise, for a representation $\tau$ of $\GL_r(E)$, we will write $\tau^c$ for the associated $c$-conjugate representation.

\vskip 5pt

 \subsection{\bf Langlands-Shahidi factors.}
The group $H_F \times \GL_r(E)$ is the Levi factor of a maximal parabolic subgroup $P_{r,F}$ of a classical group $G_F$ of the same type. 
The generator $\delta \in \Hom_F((H_F \times \GL_r(E))/ Z(G_F)^0, \G_m)$ is simply the rational character $N_{E/F} \circ \det_{\GL_r(E)}$.
We have the associated Plancherel measure $\mu(\sigma,\psi)$ if $\sigma = \pi \otimes \tau$ is a representation of $H_F \times \GL_r(E)$ and the associated LS gamma factors $\gamma(\sigma, r_i,\psi)$ if $\sigma$ is generic. 
Furthermore, one can explicate each $r_i$ in this case.  The dual group of the Levi factor is $ H_F^{\vee} \times \GL_r(\C)$ and we have
\[  r_1 = {\rm std}_{H_F^{\vee}}^{\vee} \otimes  {\rm std}_{\GL_n(\C)} 
 \] 
 where ${\rm std}$ stands for the standard representation of the relevant group. It is convenient and customary to write
 \begin{equation}  \label{E:redefine}
   \gamma(s,\pi \times \tau,\psi) \mathrel{\mathop:}= \gamma( \pi^{\vee} \otimes \tau |\det|_E^s  ,r_1,\psi).
\end{equation}
 The representation $r_2$ is given by 
\[   R  =
\begin{cases}
 {\rm Sym}^2,   \text{  if $H_F = \SO_{2n+1}$;} \\
 \wedge^2,  \text{  if $H_F = \Sp_{2n}$ or $\SO_{2n}$;} \\
 {\rm Asai}^{(-1)^n},  \text{  if $H_F = \U_n$.} \end{cases}  \]
In fact, this second $\gamma$-factor depends only on $\tau$ and we write
\begin{equation*}
   \gamma(s,\tau,R,\psi) \mathrel{\mathop:}= \gamma(\sigma \otimes  |\det|_E^{\frac{1}{2} s},r_2,\psi).
\end{equation*}
Hence, for $i=1$ or $2$, we have:
\[  \gamma(is , \pi \otimes \tau, r_i,\psi)  = \gamma( \pi \otimes \tau|\det|_E^s, r_i, \psi). \]

\vskip 5pt

\subsection{\bf The problem.}
The problem we shall consider in this section is the following. 
Starting with a supercuspidal representation $\pi$ of $H_F$, we may globalize  it to $\Pi$ using Corollary \ref{T:main2} and then use V. Lafforgue's work  \cite{La} to obtain a global Galois representation $\rho$ valued in ${^L}H_F(\overline{\Q}_l) \cong {^L}H_F(\C)$. In particular, one obtains a local L-parameter $\rho_F$ at the place $v_0$ where $k_{v_0}  = F$. It is a natural question to ask if $\rho_F$ depends on the choice of the globalization. It would, in general, since a supercuspidal representation can belong to two different Arthur packets and the resulting global Galois representations will be quite different; for example, one could be pure while the other is not. The first goal of this section is to show that, despite this, one can attach a discrete series L-parameter to a supercuspidal representation of a classical group using the work of V. Lafforgue and many others. 
 
 \vskip 5pt
 
\subsection{\bf Generic case.}
 We first consider the case of generic supercuspidal representations with respect to a fixed Whittaker datum. In \cite{L1,L2}, the second author has shown the Langlands functorial lifting from classical groups to $\GL_N$ using the converse theorem of Cogdell-Piatetski-Shapiro and the Langlands-Shahidi method (following \cite{CKPSS1, CKPSS2} in the characteristic $0$ case). As a consequence of this and the local Langlands correspondence for $\GL_N$ \cite{LRS}, one has a  map
 \begin{equation} \label{E:lift}
  \{ \text{generic supercuspidal representations of $H_F$} \}  \longrightarrow \{ \text{elliptic $L$-parameters $W_F \rightarrow  {^L}H_F$ }\} \end{equation}
  which satisfies the following property: for any irreducible generic representation $\tau$ of $\GL_r(E)$ (for any $r$) with associated L-parameter $\phi_{\tau}$, 
  \[  \gamma(s, \pi \times \tau, \psi)  =  \gamma(s, \phi_{\pi} \otimes  \phi_{\tau}, \psi), \]
  where we are using the LS gamma factor defined in (\ref{E:redefine}) on the LHS. 
 Moreover, one also knows by \cite{HL1, HL2} that
 \[  \gamma(s,  \tau, R, \psi) =  \gamma(s, \pi \otimes \tau, r_2, \psi)  = \gamma(s,  R \circ \phi_{\tau}, \psi) \]
  for any generic representation $\tau$ of $\GL_r(E)$.
  \vskip 5pt
  
   In the characteristic $0$ case, using the theory of local descent of Ginzburg-Rallis-Soudry, one can show that the map (\ref{E:lift}) is a surjection. The theory of local descent should continue to work over a local function field $F$. However,  it is presently not written up in this generality in the literature. While we could have taken the surjectivity of (\ref{E:lift}) as a working hypothesis, we prefer to use a weaker one to be described below.
  \vskip 5pt
  
 \subsection{\bf A working hypothesis.}  \label{SS:hypo}
 Suppose that  $H'_F$ is the split  $\SO_{2n+1}$ or the quasi-split unramified $\SO_{2n}$. As described in \cite{S1},  one can construct a tamely ramified regular elliptic L-parameter
\[    \phi_1:   W_F \longrightarrow  \Sp_{2n}(\C)  \quad \text{or} \quad \O_{2n}(\C)  \]
 which is irreducible as a representation of $W_F$.   
If  $H'_F = \Sp_{2n}$ so that ${H'_F}^{\vee} = \SO_{2n+1}(\C)$, one still has an ``almost irreducible" elliptic tamely ramified L-parameter 
\[  \phi_1 :  W_F \longrightarrow  \SO_{2n+1}(\C) \]
of the form
\[  \phi_1 = \chi_1 +  \phi_1' \]
 with $\chi_1$ 1-dimensional and $\phi_1':  W_F \longrightarrow \O_{2n}(\C)$  irreducible. 
We would like to know that such a $\phi_1$ is in the image of the map (\ref{E:lift}).
\vskip 5pt

 In characteristic $0$, this was shown by Savin \cite{S1} for the symplectic and orthogonal groups  by using the Debacker-Reeder construction of depth $0$ supercuspidal L-packets \cite{DR}. Unfortunately, both \cite{DR} and \cite{S1} are written in the context of characteristic $0$ local fields. However,  it seems likely that the parts which are relevant for us carry over to the characteristic $p >0$ setting (at least if $p \ne 2$). In what follows, we shall take this as a working hypothesis:
 \vskip 5pt
 
 \noindent{\bf Working Hypothesis I}: Assume that $H'_F$ is the split  $\SO_{2n+1}$ or the quasi-split unramified $\SO_{2n}$. The tamely ramified parameter $\phi_1$ described here lies in the image of (\ref{E:lift}).

\vskip 5pt
 Let $\pi_1$ be the generic supercuspidal representation of $H'_F$ which is mapped to $\phi_1$ under (\ref{E:lift}) in the symplectic and orthogonal case.
In the unitary case, we shall let $H'_F  = \GL_n$ and $\phi_1$ be any irreducible representation of $W_F$. In particular, we don't need this working hypothesis for unitary groups.

\vskip 10pt

\subsection{\bf Globalization.}
 Now, appealing to Corollary \ref{T:main2}, given any supercuspidal representation $\pi_F$ of $H_F$, there exist:
\vskip 5pt

\begin{itemize}
\item   a function field $k$ with $k_{v_0} \cong F$ at a place $v_0$;
\vskip 5pt

\item a quasi-split group $H_k$ over $k$ such that $H_{k, v_0} \cong H_F$ and $H_{k, v_1}$ is an unramified group as described in \S \ref{SS:hypo};
\vskip 5pt

\item a cuspidal representation $\Pi$ of $H_k(\A)$ such that $\Pi_{v_0} = \pi_F$,  $\Pi_{v_1} = \pi_1$ (with $\pi_1$ defined in \S \ref{SS:hypo}) and $\Pi_v$ is a constituent of a principal series representation induced from  a Borel subgroup for all other $v$. 
\end{itemize}

\vskip 5pt

Let 
\[  \rho_{\Pi} :  {\rm Gal}(k^{sep}/k) \longrightarrow   {^L}H_k(\overline{\Q}_l)  \]
be the semisimple Galois representation associated to $\Pi$ by V. Lafforgue.  The following proposition describes some key properties of $\rho_{\Pi}$.
\vskip 5pt

\begin{prop}  \label{P:deligne}
(i) The local representation $ \rho_{\Pi, v_1} $ corresponds to the Weil-Deligne representation $\phi_1$. In particular, 
the global representation $\rho_{\Pi}$ is either irreducible or the sum of a quadratic character and an irreducible self-dual representation.
\vskip 5pt

(ii)  The global representation $\rho_{\Pi}$ is $\iota_l$-pure of weight $0$ (in the sense of \cite[\S 1.2.6]{D}).

\vskip 5pt

(iii) The local representation $ \rho_{\Pi, v_0}$ corresponds to a discrete series L-parameter for $H_F$.
\end{prop}

\begin{proof}
(i) It follows by Corollary  \ref{C:plan} that, with $\rho = \rho_{\Pi}$, 
\[  \mu(s,\pi_1 \times \tau, \psi) =  \gamma(s,\rho_{v_1}^{\vee} \otimes \phi_{\tau}, \psi) \cdot     \gamma(-s,\rho_{v_1} \otimes \phi_{\tau}^{\vee},\overline{\psi}) \cdot    \gamma (2s,R \circ \phi_{\tau},\psi) \cdot \gamma(-2s,R^{\vee} \circ \phi_{\tau},\overline{\psi}).  \]
On the other hand, by Langlands-Shahidi theory \cite{Sh2, L2} and the properties of the map (\ref{E:lift}), one has
\[  \mu(s,\pi_1 \times \tau, \psi) =    \gamma(s,\phi_1^{\vee} \otimes \phi_{\tau}, \psi) \cdot   \gamma(-s, \phi_1 \otimes \phi_{\tau}^{\vee},\overline{\psi}) \cdot \gamma (2s,R \circ \phi_{\tau},\psi) \cdot \gamma(-2s,R^{\vee} \circ \phi_{\tau},\overline{\psi}).  \]
   Thus, one has
\[   \gamma(s,\rho_{v_1}^{\vee} \otimes \phi_{\tau}, \psi)  \cdot \gamma(-s,\rho_{v_1} \otimes \phi_{\tau}^{\vee},\overline{\psi})  =  \gamma(s,\phi_1^{\vee} \otimes \phi_{\tau}, \psi) \cdot  \gamma(-s,\phi_1 \otimes \phi_{\tau}^{\vee},\overline{\psi}) . \]
\vskip 5pt

In general, such an identity is not sufficient to force $\rho_{v_1}  = \phi_1$. However, in our case, $\phi_1$ is almost an irreducible representation of $W_F$, and this additional property will give this identity.

\vskip 5pt

More precisely, 
if $\phi_1$ is irreducible as a representation of $W_F$, then for any $\phi_{\tau}$ of dimension $< \dim \phi_1$, the RHS of the above equation has no zeros or poles. On the other hand, if $\rho_{v_1}$ is not an irreducible representation of $W_F$, then the LHS of the above equation will have a zero for some $\phi_{\tau}$ of dimension $< \dim \phi_1$. Thus, we deduce that $\rho_{v_1}$ is irreducible as a representation of $W_F$ as well.  Then, taking $\phi_{\tau}  = \phi_1$ in the above equation, we deduce that the RHS is zero and thus so is the LHS, which implies that $\rho_{v_1}  =  \phi_1$. 
\vskip 5pt

If $\phi_1$ is not irreducible, then $\phi_1 = \chi + \phi_1'$ with $\phi_1'$ irreducible of dimension $2n$. In this case, a similar argument as above gives $\rho_{v_1}  = \phi_1$; we leave the details to the reader. 
\vskip 10pt

\noindent (ii) This follows from (i) and \cite[Theorem VII.6]{LL} (proving a conjecture of Deligne \cite[Conjecture 1.2.10]{D}).
\vskip 10pt

\noindent (iii) It follows from (ii) and a theorem of Deligne \cite[Theorem 1.8.4]{D} that $\rho_{v_0}$ is a tempered L-parameter. Moreover, by Corollary \ref{C:plan}, one has
\[  \mu(s,\pi \times \tau, \psi) =   \gamma(s,\rho_{v_0}^{\vee} \otimes \phi_{\tau}, \psi) \cdot  \gamma(-s,\rho_{v_0} \otimes \phi_{\tau}^{\vee},\overline{\psi}) \cdot  \gamma (2s,R \circ \phi_{\tau},\psi) \cdot \gamma(-2s,R^{\vee} \circ \phi_{\tau},\overline{\psi})  \]
for any $\phi_{\tau}$.  Now if $\phi_{\tau}$ is not conjugate-self-dual, then it follows by \cite[Prop. IV.2.2(ii)]{W} that  the LHS is nonzero, and hence so is the RHS. This implies that $\rho_{v_0}$ does not contain any non-conjugate-self-dual summand.  Further, it follows by \cite[Corollary IV.1.2(ii)]{W} that the LHS has a zero of order at most $2$, which implies that $\rho_{v_0}$ is multiplicity free. Thus, $\rho_{v_0}$ is the multiplicity-free sum of conjugate-self-dual representations of $W\!D_F$ (of sign $\epsilon(H_F)$), and thus is a discrete series parameter for $H_F$.  
\end{proof}

\vskip 5pt
\subsection{\bf Independence.}
The following proposition shows that the discrete series L-parameter obtained in (iii) above is independent of the various choices used in its construction. 
 \vskip 10pt
 
\begin{prop}  \label{P:indep}
Suppose that
\vskip 5pt
\begin{itemize}
\item  $k$ and $k'$ be two global function fields such that $k_{v_0} \cong k'_{v'_0} \cong F$;
\vskip 5pt

\item $H_k$ and $H_{k'}$ are algebraic groups over $k$ and $k'$ respectively such that $H_{k, v_0} \cong H_{k', v'_0} \cong H_F$;
  \vskip 5pt

\item $\Pi$ and $\Pi'$ are two cuspidal representations such that $\Pi_{v_0}  \cong \Pi'_{v'_0} \cong \pi_F$.

\vskip 5pt

\item the associated Galois representations  $\rho_{\Pi,l}$ and $\rho_{\Pi', l'}$ are both pure of weight $0$ (where $l$ and $l'$ are any two prime numbers different from $p$).
\end{itemize}
\vskip 5pt

\noindent Then the local representations $\rho_{\Pi, v_0, l}$ and $\rho_{\Pi', v'_0, l'}$ are equivalent as L-parameters of $H_F$.
\end{prop}

\begin{proof}
By the previous proposition, we know that $\rho_{\Pi, v_0, l}$ and $\rho_{\Pi', v'_0, l'}$ are both discrete series L-parameters of $H_F$. Moreover, by Corollary \ref{C:plan}, we have
\[  \gamma(-s, \rho_{\Pi, v_0,l} \otimes \phi_{\tau}^{\vee},\overline{\psi}) \cdot \gamma(s, \rho_{\Pi, v_0,l}^{\vee} \otimes \phi_{\tau}, \psi) = \gamma(-s, \rho_{\Pi', v_0', l'} \otimes \phi_{\tau}^{\vee},\overline{\psi}) \cdot \gamma(s, \rho_{\Pi', v'_0, l'}^{\vee} \otimes \phi_{\tau}, \psi). \]
for any irreducible representation $\phi_{\tau}$ of $W_E$.   By \cite[Lemma 12.3]{GS}, this implies that $\rho_{\Pi, v_0, l}$ and $\rho_{\Pi', v'_0, l'}$ are equivalent as L-parameters of $H_F$.
\end{proof}
 \vskip 10pt

\subsection{\bf $L$-parameters of supercuspidal representations.}
To summarise,we have shown:
\vskip 5pt

\begin{thm}  \label{T:L}
We assume the Working Hypothesis. Let $\tau$ be a supercuspidal representation of $\GL_r(F)$ (for any $r$) with associated L-parameter $\phi_{\tau}$. 
\vskip 5pt

(i) For each prime number $l \ne p$, there is a map
\[ \mathcal{L}_l:  \{\text{supercuspidal representations of $H_F$} \}  \longrightarrow 
\{ \text{elliptic L-parameters $W\!D_F  \longrightarrow {^L}H_F$} \}. \]
Write $\phi_\pi = \mathcal{L}_l(\pi)$ for the corresponding Langlands parameter of a representation $\pi$.
\vskip 5pt

\noindent (ii) Suppose $\pi$ is a supercuspidal generic representation of $H_F$. The map $\mathcal{L}_l$  has the property that
\[   L(s,\pi \times \tau) =   L(s,\phi_\pi \otimes \phi_{\tau}) \quad \text{ and } \quad  \varepsilon(s,\pi \times \tau,\psi) = \varepsilon(s,\phi_\pi \otimes \phi_\tau,\psi).  \]

\vskip 5pt

\noindent (iii) The map $\mathcal{L}_l$  has the property that
\[  \mu(s, \pi \times \tau, \psi) =  \gamma(s,\phi_\pi^{\vee} \otimes \phi_{\tau}, \psi) \cdot   \gamma(-s,\phi_\pi \otimes \phi_{\tau}^{\vee},\overline{\psi}) \cdot  \gamma (2s, R \circ \phi_{\tau},\psi) \cdot \gamma(-2s, R^{\vee} \circ \phi_{\tau},\overline{\psi}).  \] 

\vskip 5pt

(iv) Moreover, $\mathcal{L}_l$ is characterized by properties $(ii)$ and $(iii)$ and is independent of $l$ (so that we simply denote it by $\mathcal{L}$).
\end{thm}

 \vskip 10pt
 It is natural to ask if the map $\mathcal{L}$ defined in Theorem \ref{T:L} can be extended to all irreducible representations of $H_F$.  
 The key step is to extend it to the set of irreducible discrete series representations. If one can do this, then an application of the Langlands classification theorem would extend it to all irreducible representations.    To do so it is necessary to use another deep result, namely the classification of discrete series representations of classical groups in terms of supercuspidal ones due to Moeglin-Tadi\'c \cite{M, MT}. The results of \cite{M, MT} were obtained under a basic assumption (BA). In the next theorem, we shall verify the assumption (BA) in \cite{MT}.
 
 \vskip 5pt
 
 \subsection{\bf Reducibility of Generalized Principal Series.}
 Using the above results, we can obtain the reducibility points of the generalized principal series representations of quasi-split classical groups induced from supercuspidal representations of maximal parabolic subgroups. More precisely, 
 let $P  = M\cdot N \subset G$ be a maximal parabolic subgroup of a classical group $G$ over $F$, and suppose that its Levi factor $M$ is isomorphic to $\GL_r(E)  \times 
H_F$.  Let $\tau \otimes \pi$ be a unitary supercuspidal representation of $M$ and consider the generalized principal series representation
\[  {\rm I}(s, \tau \otimes \pi)  = {\rm Ind}_P^G  \tau  |\det|_E^s  \otimes \pi \]
for $s \in \R$, with its associated Plancherel measure $\mu(s, \tau \otimes \pi,\psi)$. 
We shall make use of the following well-known properties of the Plancherel measure for supercuspidal inducing data due to Harish-Chandra and Silberger (see \cite[Pg. 296, Remark 2 and Lemma 5.4.2.4]{Si1} and  \cite[Lemmas 1.1 and  1.2]{Si2}; see also \cite{Ca, S2, W}). 
\vskip 5pt

\begin{lemma}  \label{L:silberger}
(a)   If $\tau^{\vee} \ncong \tau^c$, then ${\rm I}(s, \tau \otimes \pi)$ is irreducible for all $s \in \R$.
\vskip 5pt

\noindent (b) If $\tau^{\vee} \cong \tau^c$, then ${\rm I}(0, \tau \otimes \pi)$ is reducible if and only if $\mu(0, \tau \otimes \pi, \psi)$ is nonzero, in which case, ${\rm I}(s, \tau \otimes \pi)$ is irreducible for all real numbers $s \ne 0$ and $\mu(s, \tau \otimes \pi, \psi)$ is holomorphic at all $s  \in \R$. 
\vskip 5pt

\noindent (c) If $\tau^{\vee} \cong \tau^c$, but $\mu(0,\tau \otimes \pi, \psi) =0$ so that ${\rm I}(0, \tau \otimes \pi)$ is irreducible, then  ${\rm I}(s_0, \tau \otimes \pi)$ is reducible for $s_0 >0$ if and only if $\mu(s, \tau \otimes \pi, \psi)$ has a pole at $s = s_0$. Moreover, there is a unique $s_0 >0$ such that reducibility occurs, and at this point of reducibility,  the pole of $\mu(s, \tau \otimes \pi, \psi)$  is simple. Further, $\mu(s, \tau \otimes \pi, \psi)$ is nonzero for any real $s \ne 0$. 
\vskip 5pt

\noindent In particular, if $\tau^{\vee} \cong \tau^c$, there is a unique $s_0 \geq 0$ such that ${\rm I}(s_0, \tau \otimes \pi)$ is reducible.
\end{lemma}
\vskip 5pt 

 The above properties of the Plancherel measure imply the following proposition:
\vskip 5pt

\begin{prop}  \label{P:noholes}
Let $\phi_{\pi} = \mathcal{L}(\pi)$ be the L-parameter of the supercuspidal representation $\pi$ supplied by Theorem \ref{T:L}.  
Then $\phi_{\pi}$ is ``sans trou" (without holes) in the sense of Moeglin \cite{M}. More precisely,
for any irreducible representation $\rho$ of $W_E$, such that $\det \rho$ is unitary,  let 
\[  {\rm Jord}_{\rho}(\phi_{\pi}) = \{  a \in \mathbb{N}:  \rho \otimes S_a  \subset \phi_{\pi} \}, \]
where $S_a$ denotes the $a$-dimensional irreducible representation of ${\rm SL}_2(\C)$.  Then
$ {\rm Jord}_{\rho}(\phi_{\pi})$ can be nonempty only if $\rho$ is conjugate-self-dual of some sign $\epsilon(\rho)$, in which case
 all elements in $ {\rm Jord}_{\rho}(\phi_{\pi})$ are of the same parity: $a \in  {\rm Jord}_{\rho}(\phi_{\pi})$ is odd if and only if $\epsilon(\rho)  = \epsilon(H_F)$.   Then for all $\rho$ such that $ {\rm Jord}_{\rho}(\phi_{\pi})$ is nonempty and  any integer $a> 2$,
\[   a \in {\rm Jord}_{\rho}(\phi_{\pi})  \Longrightarrow a-2  \in {\rm Jord}_{\rho}(\phi_{\pi}). \]
\end{prop}

\vskip 5pt

\begin{proof}
Suppose that $\rho$ is conjugate-self-dual and ${\rm Jord}_{\rho}(\phi_{\pi})$ is nonempty. 
Let $\tau_{\rho}$ be the supercuspidal representation of $\GL_r(E)$ (with $r = \dim \rho$) with L-parameter $\rho$, and consider the family of induced representations 
${\rm I}(s, \tau_{\rho} \otimes \pi)$. Recall that 
\[  \mu(s, \tau_{\rho} \otimes \pi, \psi) = \gamma(s,  \rho \otimes \phi_{\pi}^{\vee},\psi) \cdot  \gamma(-s,  \rho^\vee \otimes \phi_{\pi},\overline{\psi}) \cdot \gamma(2s, R \circ \rho,\psi) \cdot \gamma(-2s, R^{\vee} \circ \rho, \overline{\psi}). \]
The RHS is essentially a ratio of products of local L-functions and epsilon factors, and the part which could contribute poles or zeros in $s \geq 0$ is:
\[ \left(  \prod_{a \in   {\rm Jord}_{\rho}(\phi_{\pi})}  \frac{L(\frac{a+1}{2} -s, \rho^{\vee} \otimes \rho)}{L(\frac{a-1}{2} -s, \rho^{\vee} \otimes \rho)}  \right)  \cdot 
\left( \frac{L(1-2s, R^{\vee} \circ \rho) \cdot L(1+2s, R \circ \rho)}{L(2s, R \circ \rho) \cdot L(-2s, R^{\vee} \circ \rho)}\right) . \]
 From this, we see that  the poles and zeros of $ \mu(s, \tau_{\rho} \otimes \pi, \psi)$ occur at the following points:
 \begin{itemize}
 \item for $\epsilon(\rho)  = \epsilon(H_F)$:   
 \[ \text{Poles at  $\frac{a+1}{2}$  with  $a \in {\rm Jord}_{\rho}(\phi_{\pi})$ but $a+2 \notin {\rm Jord}_{\rho}(\phi_{\pi})$,} \]
 and
 \[  \text{Zeros at $\frac{a-1}{2}$ with $a \in {\rm Jord}_{\rho}(\phi_{\pi})$ but $a-2 \notin {\rm Jord}_{\rho}(\phi_{\pi})$.} \]
 \vskip 5pt
 
 \item for $\epsilon(\rho)= - \epsilon(H_F)$:
 \[ \text{Poles at  $\frac{a+1}{2}$  with  $a=0$ or $a  \in {\rm Jord}_{\rho}(\phi_{\pi})$, but $a+2 \notin {\rm Jord}_{\rho}(\phi_{\pi})$,}   
  \]
 and
 \[  \text{Zeros at  $0$ and $\frac{a-1}{2}$ with $a \in {\rm Jord}_{\rho}(\phi_{\pi})$ but $a-2 \notin {\rm Jord}_{\rho}(\phi_{\pi})$.} \]
 \end{itemize}
 Hence, if we set
 \[  a_{\rho}(\pi)  = {\rm max}  \,  {\rm Jord}_{\rho}(\phi_{\pi}), \]
 then $\mu(s, \tau_{\rho} \otimes \pi, \psi)$ has a pole at $s = (a_{\rho}(\pi) +1)/2 \geq 1$, so that
$s_0 =  (a_{\rho}(\pi) +1)/2$ must be the unique reducibility point of $I(s, \tau_{\rho} \otimes \pi)$ with $s \geq 0$. 
In particular, by Lemma \ref{L:silberger}, we must have $\mu(0, \tau_{\rho} \otimes \pi, \psi)  = 0$ but $\mu(s,  \tau_{\rho} \otimes \pi,\psi) \ne 0$ for any $s >0$. 
Hence, we conclude that in both cases above, for all $a >  2$, 
\[  a \in {\rm Jord}_{\rho}(\phi_{\pi})  \Longrightarrow a-2  \in {\rm Jord}_{\rho}(\phi_{\pi}). \]
This proves the proposition. 
 \end{proof}

 Now we have the following theorem which establishes the basic assumption (BA) in \cite{MT}.
\vskip 5pt

\begin{thm}  \label{T:redu}
Let $\phi_{\pi} = \mathcal{L}(\pi)$ be the L-parameter of $\pi$ supplied by Theorem \ref{T:L}, and let $\phi_{\tau}$ be the L-parameter of a unitary supercuspidal representation $\tau$ of $\GL_r(E)$.   The representation ${\rm I}(s_0, \tau \otimes \pi)$ is reducible if and only if $\tau^{\vee} \cong \tau^c$ and one of the following holds:\vskip 5pt

\begin{itemize}
\item[(i)]  $ s_0 = \frac{a_{\tau}(\pi) +1}{2} \geq 1$, with
$a_{\tau}(\pi) =  {\rm max}  \,  {\rm Jord}_{\phi_{\tau}}(\phi_{\pi})$, if   ${\rm Jord}_{\phi_{\tau}}(\phi_{\pi})$ is nonempty;
\vskip 5pt

\item[(iii)]  $s_0=1/2$ if    ${\rm Jord}_{\phi_{\tau}}(\phi_{\pi})$ is empty and $L(2s, R \circ  \phi_{\tau})$ has a pole at $s = 0$ (i.e. $\epsilon(\phi_{\tau}) = - \epsilon(H_F)$);
\vskip 5pt

\item[(i)] $s_0 =0$ if ${\rm Jord}_{\phi_{\tau}}(\phi_{\pi})$ is empty and  $L(2s, R \circ  \phi_{\tau})$ is holomorphic at $s = 0$ (i.e. $\epsilon(\phi_{\tau}) = \epsilon(H_F)$).
\end{itemize}
 \end{thm}
\vskip 5pt

\begin{proof}
This follows immediately from Lemma \ref{L:silberger} and the proof of Proposition \ref{P:noholes}. 
  \end{proof}
\vskip 5pt

We note that such a theorem was first shown by Shahidi \cite{Sh2} for general quasi-split groups and generic supercuspidal inducing data, in which case the reducibility points are at $0$, $1/2$ or $1$.

 \vskip 10pt
  
 \subsection{\bf Results of Moeglin-Tadi\'c}
 
Theorem \ref{T:redu} renders the results of Moeglin-Tadi\'c \cite{M, MT} unconditional. This places us in a position to extend the map $\mathcal{L}$ in Theorem \ref{T:L} from supercuspidal representations to discrete series representations. The  procedure (due to Moeglin-Tadi\'c) for extending $\mathcal{L}$ from supercuspidal representations to discrete series representations  has been explained in great detail and clarity in \cite[\S 7]{CKPSS2}. Let us give a brief description here, following \cite[\S 7]{CKPSS2} closely. 
  
  \vskip 5pt
  
  Moeglin-Tadi\'c showed that any nonsupercuspidal discrete series representation $\pi$ can be uniquely expressed as a subquotient of an induced representation of the form
  \begin{equation}  \label{E:MT}
    {\rm Ind}_{P_F}^{H_F}  (\bigotimes_{i \in S} \delta_i)  \otimes  (\bigotimes_{j \in T} \delta_j')  \otimes \pi_0 , \end{equation}
  where
  \begin{itemize}
   \item $\pi_0$ is a supercuspidal representation of a smaller classical group of the same type as $H_F$.
   \vskip 5pt
   
  \item For $i\in S$, $\delta_i$ is the generalized Steinberg representation of $\GL_{k_i}$ contained in the induced representation
  \[   \tau_i |-|^{-(b_i -1)/2}  \times \cdots \times  \tau_i|-|^{(a_i -1)/2}, \]
  where $a_i > b_i > 0$ are positive integers of the same parity and  $\tau_i$ is a  supercuspidal representation which is conjugate-self-dual with sign $\epsilon(H_F) \cdot (-1)^{a_i-1}$.
  \vskip 5pt
  
  \item For $j \in T$, $\delta'_j$ is the generalized Steinberg representation of $\GL_{k_j'}$ contained in the induced representation
  \[ \tau'_j |-|^{(c'_j+1)/2  }  \times \cdots  \times \tau'_j |-|^{(a_j'-1)/2} \]
  where $\tau'_j$ is a conjugate-self-dual  supercuspidal representation and $c_j'  \in \{ 1,2 \}$ has the same parity as $a_j'$ with $a_j' \geq c_j' +2$.
  Moreover, the $\tau_j'$'s are pairwise distinct and  
  \begin{alignat}{2}
   &\text{$a_j'$ odd}&& \Longrightarrow \text{$L(s, \phi_{\tau'_j}^{\vee} \otimes \phi_{\pi_0})$ has a pole at $s  =0$.} \notag \\
  &\text{$a_j'$ even}  &&\Longrightarrow \text{$L(s,  r_2 \circ \phi_{\tau'_j})$ has a pole at $s = 0$.}\notag  \end{alignat}
  In particular, if $a_j'$ is odd, then $\phi_{\tau'_j}$ is a summand in $\phi_{\pi_0}$ and 
  \[  \phi_{\pi_0}   -  \bigoplus_{j \in T: \text{$a_j'$ odd}}  \phi_{\tau'_j}  \]
  is an elliptic L-parameter for a smaller classical group of the same type as $H_F$.
    \end{itemize}
    \vskip 5pt
    
  Given this,  one can define the L-parameter of $\pi$ by:
  \begin{align}  \label{A:para}
    \mathcal{L}(\pi) =&\left( \bigoplus_{i \in S} \phi_{\tau_i} \otimes (S_{a_i} \oplus S_{b_i}) \right) \oplus \left(\bigoplus_{j \in T: \text{$a_j'$ even}}  \phi_{\tau'_j} \otimes S_{a'_j} \right)  \notag \\
  &  \oplus 
  \left( \bigoplus_{j \in T: \text{$a_j'$ odd}}   \phi_{\tau'_j} \otimes S_{a_j'} \right) \oplus \left( \phi_{\pi_0}   -  \bigoplus_{j \in T: \text{$a_j'$ odd}}  \phi_{\tau'_j} \right).  
  \end{align}
 It was shown in \cite{M} and \cite{MT} that this is a discrete series L-parameter for $H_F$, i.e., it is multiplicity-free. 
 
 \vskip 5pt
 
To see  that the Plancherel measure $\mu(s, \pi \times \tau,\psi)$ can be expressed in terms of $\mathcal{L}(\pi)$ as in Theorem \ref{T:L}(iii), we note that  by the multiplicativity property of Plancherel measures, $\mu(s, \pi \times \tau,\psi)$ depend only on the supercuspidal support of $\pi \otimes \tau$. Consider the representation  $\phi$ of $W_F$  associated to the supercuspidal support of the induced representation (\ref{E:MT}). Setting
\[  \rho_a =  \bigoplus_{i = 0}^{a-1}  |-|^{\frac{a-1}{2} - i }  \]
to be the L-parameter of the trivial representation of $\GL_a(F)$ for simplicity, we see that
\begin{align}  \label{A:para2}
  \phi = &\left( \bigoplus_{i \in S}  \phi_{\tau_i} \otimes (\rho_{a_i} \oplus \rho_{b_i}) \right)  \oplus \left( \bigoplus_{j\in T: \text{$a_j'$ even}}  \phi_{\tau'_j} \otimes \rho_{a'_j} \right)  \notag \\
   & \oplus  \left( \bigoplus _{j \in T: \text{$a_j'$ odd}}   \phi_{\tau'_j} \otimes \rho_{a'_j}  \right)  \oplus \left( \phi_{\pi_0} - \bigoplus_{j \in T: \text{$a_j'$ odd}}  \phi_{\tau'_j} \right). 
   \end{align}
 By multiplicativity and Theorem \ref{T:L}(iii)( for supercuspidal representations), one has
 \begin{equation}  \label{E:mu}
 \mu(s, \pi \times \tau,\psi) =  \gamma(s,\phi^{\vee} \otimes \phi_{\tau}, \psi) \cdot   \gamma(-s,\phi \otimes \phi_{\tau}^{\vee},\overline{\psi}) \cdot  \gamma (2s, R \circ \phi_{\tau},\psi) \cdot \gamma(-2s, R^{\vee} \circ \phi_{\tau},\overline{\psi}).  \end{equation}
 Comparing (\ref{A:para}) and (\ref{A:para2}), and noting that 
 \[  \gamma(s,  \Sigma \otimes  S_a, \psi)  = \gamma(s, \Sigma \otimes \rho_a, \psi)  \]
 for any representation $\Sigma$ of $WD_F$, we deduce that the RHS of (\ref{E:mu}) is equal to the same expression with $\phi$ replaced by $\mathcal{L}(\pi)$, as desired. 
 \vskip 10pt

  \subsection{\bf The LLC}
 In view of the above discussion and using the Langlands classification theorem, we obtain:
 \vskip 5pt
 
  \begin{thm}  \label{T:L2}
 Assume Working Hypothesis I (which is not needed when $H_F$ is a unitary group).   There is a  map
\[ \mathcal{L}:  \{\text{irreducible  smooth representations of $H_F$} \}  \longrightarrow 
\{ \text{L-parameters $W\!D_F  \longrightarrow {^L}H_F$} \}, \]
satisfying the following properties:
\vskip 5pt

\begin{itemize}
\item[(i)]  Writing $\phi_\pi = \mathcal{L}(\pi)$ for the corresponding Langlands parameter of a representation $\pi$, we have:
\[
\text{ $\pi$ is a discrete series representation $\Longleftrightarrow$ $\phi_{\pi}$ is a discrete series L-parameter,} \]
\[  \text{ $\pi$ is a tempered representation $\Longleftrightarrow$ $\phi_{\pi}$ is a tempered L-parameter.} \]

\item[(ii)] The map $\mathcal{L}$ is compatible with the Langlands classification theorem. More precisely, suppose $\pi$ is the unique Langlands quotient of a standard module ${\rm Ind}_{P_F}^{H_F} \tau$ where $P_F = M_F N_F$ is a parabolic subgroup and $\tau$ is an essentially tempered representation of the Levi factor $M_F$. Then
$ \phi_{\pi}$ is given by the composite
\[
\begin{CD}
  W\!D_F  @>\phi_{\tau}>>   {^L}M_F @>>> {^L}H_F  \end{CD} \]
where the first arrow is the L-parameter $\phi_{\tau}$ of $\tau$.

\vskip 5pt

\item[(iii)] Suppose $\pi$ is an irreducible generic representation of $H_F$. Then
\[   L(s,\pi \times \tau) =   L(s,\phi_\pi \otimes \phi_{\tau}) \quad \text{ and } \quad \varepsilon(s,\pi \times \tau,\psi) = \varepsilon(s,\phi_\pi \otimes \phi_\tau,\psi)  \] 
where $\tau$ is any irreducible representation of $\GL_r(F)$ (for any $r$) with L-parameter $\phi_{\tau}$. 
\vskip 5pt

\item[(iv)] For any $\pi$, one has
\[  \mu(s, \pi \times \tau, \psi) =   \gamma(s,\phi_\pi^{\vee} \otimes \phi_{\tau}, \psi) \cdot  \gamma(-s,\phi_\pi \otimes \phi_{\tau}^{\vee},\overline{\psi}) \cdot  \gamma (2s,R  \circ \phi_{\tau},\psi) \cdot \gamma(-2s, R^{\vee} \circ \phi_{\tau},\overline{\psi}).  \] 
\end{itemize}
\vskip 5pt

Moreover, $\mathcal{L}$ is characterized by properties  (i), (ii), (iii) and (iv).
 \end{thm}
 \vskip 10pt

 \vskip 10pt
 
 \subsection{\bf Some questions.}\label{questions}
 Naturally, we are led to ask the following questions:
 \vskip 5pt
 
 \begin{itemize}
 \item Is the map $\mathcal{L}$ surjective?
 
 \item Are the fibres of $\mathcal{L}$ finite?  
 
 \item If $\phi$ is a tempered L-parameter, is there a generic representation in its fiber under $\mathcal{L}$? This is the tempered L-packet conjecture of Shahidi.
 
 \item Is there a refined parametrisation of the fibres of $\mathcal{L}$ in terms of characters of a certain component group? 
 \end{itemize}
 \vskip 10pt
 
  Note that if one has the local descent results in positive characteristic, Working Hypothesis I would not be needed in Theorem \ref{T:L2} and the surjectivity of $\mathcal{L}$ would also follow. 
 \vskip 10pt
  
 We remark that in a recent preprint \cite{GV}, R. Ganapathy and S. Varma have used the Deligne-Kazhdan theory of close local fields to deduce the local Langlands correspondence for split classical groups in characteristic $p >0$ from the case of characteristic $0$.  Their map satisfies the properties of the above theorem and thus agree with our map $\mathcal{L}$; moreover, the above questions all have affirmative answers.
  We should also mention that in an ongoing work, A. Genestier and V. Lafforgue are trying to establish the local Langlands correspondence by a local analog of \cite{La}, and in particular to obtain the map $\mathcal{L}$ as in the theorem, for a general reductive group $G$ in characteristic $p>0$. 
  Their more geometric methods  should complement and perhaps go further than those of this paper.
\vskip 10pt

 \section{\bf Application of the Trace Formula}  \label{S:LLC2}
 We continue to assume that $H_F$ is a quasi-spit classical group over the local function field $F$.
 In this section, we consider an alternative way of extending the map $\mathcal{L}$ of Theorem \ref{T:L} from  supercuspidal representations to discrete series representations. Instead of appealing to the deep results of Moeglin-Tadi\'c, we shall use a global-to-local argument similar to the construction of $\mathcal{L}$ in the  supercuspidal case.
 \vskip 5pt

 The construction of the map $\mathcal{L}$ for supercuspidal representations  would also apply to discrete series representations if one can globalize discrete series representations. The Poincar\'e series argument used in our proof of Theorem \ref{T:main} only works for supercuspidal representations. However, in the characteristic 0 situation, one can use the Arthur-Selberg trace formula to globalize discrete series representations. Indeed, it suffices to have a weak version of ``limit multiplicity formula", such as that shown by Clozel in \cite{C}. Unfortunately, in positive characteristic, the local theory of invariant harmonic analysis and the global theory of the trace formula are not as fully developed as in the characteristic 0 case. As such, we shall make an additional working hypothesis (a simple trace formula) which we shall describe in a moment.
 
 \vskip 5pt
 
 \subsection{\bf Pseudo-Coefficients.}
 In order to detect non-supercuspidal discrete series representations using the trace formula, we need the notion of pseudo-coefficients.
 It has been shown by Henniart-Lemaire \cite{HLe} that
   any irreducible discrete series representation $\pi$ of $H_F$ has a pseudo-coefficient $f_{\pi}$. More precisely, $f_{\pi}  \in C^{\infty}_c(H_F)$ has the property that
  \[ {\rm Tr}  \,  \sigma(f_\pi)  = \begin{cases}
  1 \text{ if $\sigma \cong \pi$,} \\
  0, \text{ for any irreducible tempered representation $\sigma \ncong \pi$.} \end{cases} \]
 It follows that such an $f_{\pi}$ satisfies the following additional properties:
 \begin{itemize}
  \item the orbital integral of $f_{\pi}$ vanishes on all non-elliptic regular semisimple elements;
  
  \item $f_{\pi}(1)$ is equal to the formal degree of $\pi$ (with respect to an appropriate Haar measure) and thus is nonzero;
  
  \item for any standard module $I_P(\tau) = {\rm Ind}_{P(F)}^{H_F(F)} \tau$ with $P \neq H_F$ a proper parabolic subgroup, ${\rm Tr}\,( I_P(\tau)(f_{\pi}))  = 0$. 
   \end{itemize}
   \vskip 5pt

   If $\pi$ is supercuspidal, one can simply take $f_{\pi}$ to be a matrix coefficient of $\pi$ with $f_{\pi}(1)  \ne 0$. 
 Then such a pseudo-coefficient is a very cuspidal  function in the sense of  \cite[Pg. 133, Definition 5.1.4]{Lau2}. More precisely, it satisfies: for any proper parabolic subgroup $P = MN \subset H_F$ and a special maximal compact subgroup $K$ in good relative position to $P$,  
 \[   f_{\pi}^P  (m) := \delta_P(m)^{1/2} \cdot \int_{N(F)} \int_K  f_{\pi}(k^{-1} mn k)  \, dk \, dn  = 0 \]
 as a function on $M(F)$.  
 
 \vskip 5pt
 
   We also note the following lemma:
  \vskip 5pt
  
 \begin{lemma}  \label{L:supp}
 Let $\pi$ be a discrete series representation of $H_F$ with pseudo-coefficient $f_{\pi}$. If $\sigma$ is an irreducible representation of $H_F$ such that ${\rm Tr} \, \sigma(f_{\pi}) \ne 0$, then $\sigma$ and $\pi$ have the same supercuspidal support.  In particular, for any irreducible representation $\tau$ of $\GL_n(F)$, one has an equality of Plancherel measures:
 \[  \mu(s, \pi \times \tau, \psi)  =  \mu(s,\sigma \times \tau, \psi). \]
\end{lemma}
\vskip 5pt

\begin{proof}
If $\sigma \ncong \pi$, then $\sigma$ is nontempered and can be written as a finite $\Z$-linear combination of standard modules, all of whose irreducible subquotients have the same supercuspidal support.  Since the trace of $f_{\pi}$ vanishes on any standard module induced from a proper parabolic subgroup as well as any tempered representation different from $\pi$, one of the standard modules which intervene in the above linear combination must be $\pi$. Thus $\sigma$ has the same supercuspidal support as $\pi$.  
\end{proof}

   \vskip 10pt
 
 \subsection{\bf Another working hypothesis.}
  Now let $k$ be a global function field and $H_k$ a connected semisimple group over $k$ (for simplicity). 
  We shall formulate another working hypothesis which is basically a simple trace formula.
  \vskip 5pt
  
  \noindent{\bf \underline{Working Hypothesis II}}  
    \vskip 5pt
    
    Let $T$ be a nonempty finite set of places of $k$. Suppose that  $f = \prod_v f_v \in C^{\infty}_c(H_k(\A))$ is  such that for $v \in T$,   $f_{v}$ is a matrix coefficient of a supercuspidal representation $\pi_{v}$ with $f_{v}(1)  =1$. For such a test function $f$, consider the kernel function 
 \[  K_f(x,y) = \sum_{\gamma \in H_k(k)} f( x^{-1} \gamma y)  \]
 for the right translation action $R(f)$ on $L^2(H(k) \backslash H(\A))$. Then $K_f(x,y)$ is integrable on the diagonal and (at least when $T$ is sufficiently large) one has a spectral and geometric expansion:
 \[     \sum_{\text{cuspidal $\Pi$}}  Tr \, \Pi(f)  = \int_{H(k) \backslash H(\A)} K_f(x,x) \, dx  =   \sum_{ \{\gamma\} }  a_{\gamma} O_{\gamma}(f)  \]
 where the sum over $\gamma$ runs over conjugacy classes of elliptic semisimple elements in $H_k(k)$, $a_{\gamma} \ne 0$ are some nonzero constants and $O_{\gamma}(f)$ is the orbital integral of $f$ over the conjugacy class of $\gamma$.  
  \vskip 10pt

  In characteristic $0$, the hypothesis follows from the work of Arthur. For global function fields, the hypothesis was established by Laumon for $H = \GL_n$ in \cite[Chapters 9 and 10]{Lau3} and a variant was used for general $H$ by Gross in \cite[\S 5]{Gr}. One certainly hopes that Laumon's proof  would extend to general groups $H$. This is not the right place to verify this, but let us make a few comments. The proof of the integrality of $K_f(x,x)$ is given in \cite[Theorem 10.2 and \S10.4]{Lau3}: one would imagine that essentially the same proof should work for general groups $H$. The fact that   only cuspidal representations intervene on the spectral side is because we have used the matrix coefficient of a supercuspidal representation at places in $T$. The main difficulty, due to the non-perfectness of $k$, is the geometric expansion which is dealt with in \cite[\S 10.6-10.9]{Lau3}; for example, one would need the important \cite[Theorem 10.7.6]{Lau3}. The details of this geometric expansion need to be verified for general $H$.

  \subsection{\bf Globalisation of discrete series.}
 Using the above Working Hypothesis II, we can demonstrate the following result, which is a weak version of a result of Clozel \cite{C} in characteristic $0$ (itself a weak version of the so-called ``limit multiplicity formula"):
 \vskip 5pt
 
 \begin{prop}
 Let $k$ be a global function field and $H_k$ a connected semisimple group over $k$ (for simplicity). Let $S \cup T$ be a disjoint union of  finite sets of places of $k$ with $S$ and $T$ nonempty and $T$ sufficiently large. Suppose we are given discrete series representations $\pi_v$ of $H_k(k_v)$ for each $v \in S$ and a supercuspidal representation $\pi_{v_1}$ of $H_k(k_{v_1})$ for all $v_1 \in T$. Then there exists a cuspidal representation $\Pi$ of $H_k(\A)$ such that 
 \begin{itemize}
 \item  for all $v_1 \in T$, $\Pi_{v_1} \cong \pi_{v_1}$;
 
  \item  for all $v \in S$, 
 \[  {\rm Tr} \, \Pi_v(f_{\pi_v})   \ne 0, \]
 where $f_{\pi}$ is a pseudo-coefficient for $\pi_v$. Hence, $\Pi_v$ and $\pi_v$ have the same supercuspidal support.
 \end{itemize}
   \end{prop}
 
 Observe that for $v \in S$, we do not assert, nor do we know, that $\Pi_v \cong \pi_v$. Thus, $\Pi$ is not exactly a globalization of $\otimes_{v \in S} \pi_v$.  However,   Lemma \ref{L:supp} implies  that $\Pi_v$ and $\pi_v$ have the same supercuspidal support for $v \in S$, so one might call $\Pi$ a ``pseudo-globalization" of $\otimes_{v \in S} \pi_v$ (obtained as a consequence of using a pseudo-coefficient). 
 Moreover, we do not care about the local components of $\Pi$ outside the set $S \cup T$ (because we have the  stability of Plancherel measures as in Corollary~\ref{plancherelstability}).
 \vskip 5pt
 
 \begin{proof}
  To apply the trace formula supplied by Working Hypothesis II, we need to specify a test function 
 $f = \prod_v f_v \in C^{\infty}_c(H_k(\A))$:
 \vskip 5pt
 
 \begin{itemize}
 \item for $v_1 \in T$, we take $f_{v_1}$ to be a matrix coefficient of $\pi_{v_1}$ with $f_{v_1}(1)  =1$;
 \item for $v \in S$, we take $f_v$ to be a pseudo-coefficient $f_{\pi_v}$ for the discrete series representation $\pi_v$;
 \item for some fixed $v_2 \notin S \cup T$, we take $f_{v_2}$ to be the characteristic function of an open compact subgroup $J \subset H_k(k_{v_2})$;
 \item for all other $v$, we take $f_v$ to be the characteristic function of a (hyper)special maximal compact subgroup. 
 \end{itemize}
 \vskip 5pt
 
 Now we apply the trace formula in Working Hypothesis II to this test function $f$. 
  On the geometric side, the sum of elliptic semisimple orbital integrals 
 \[    \sum_{ \{\gamma\} }  a_{\gamma} O_{\gamma}(f)  \]
   is a finite sum. Thus, if we shrink the open compact subgroup $J \subset H_k(k_{v_2})$  to a sufficiently small neighbourhood of $1$, we see that the only term which contributes to the geometric side of the trace formula is the one given by $\gamma  =1$. Then the geometric side is equal to
 \[ a_1 \cdot  f(1)  \ne 0. \]
Thus, invoking the spectral side, we conclude that for this particular $f$, there exists a cuspidal representation $\Pi$ such that
\[  {\rm Tr} \, \Pi(f)  = \prod_v  {\rm Tr} \, \Pi_v (f_v)  \ne 0. \]
By the properties of the local test functions $f_v$,  we see that this $\Pi$ will satisfy the requirements of the proposition. 
  \end{proof}
 \vskip 5pt
 We have assumed that $H_k$ is semisimple for simplicity. The case of reductive $H_k$ with anisotropic center is similarly handled by working with a fixed central character, with some care needed in globalizing the central character, as we discussed in \S \ref{S:corollary}; we omit the details here.  
 
 \vskip 10pt

 \subsection{\bf Definition of $\mathcal{L}$.}
 Now we can define the extension of the map $\mathcal{L}$ to all discrete series representations. Given the local field $F$ and a classical group $H_F$ over $F$, choose a global field $k$ such that $k_{v_0} \cong F$ and a classical group $H_k$ over $k$ such that $H_{k, v_0} \cong H_F$.
 We consider a finite set $S \cup T$ of places of $k$ with $S = \{ v_0 \}$ and $T$ sufficiently large.
 Given a discrete series representation $\pi$ of $H_F$, we apply the proposition with
 \vskip 5pt
 
 \begin{itemize}
 \item  $\pi_{v_0}  = \pi$; 
 \item  for all $v_1 \in T$, $\pi_{v_1} = $ the supercuspidal representation $\pi_1$ with L-parameter $\mathcal{L}(\pi_1)$ equal to the L-parameter $\phi_1$ in Working Hypothesis I.
 \end{itemize}
 Then the proposition provides a cuspidal $\Pi$ such that
 \[  \Pi_{v_1} = \pi_1 \,  \text{for all $v_1 \in T$,} \quad \text{and}  \quad {\rm Tr} \, \Pi_{v_0}(f_{\pi})  \ne 0. \]
 Note that we do not know whether $\Pi_{v_0}$ is isomorphic to $\pi$. However, if we believe in various standard conjectures in the theory of automorphic forms, we would expect that $\Pi_{v_0}$ is tempered, and thus is isomorphic to $\pi$.   
 \vskip 10pt

 We now consider the Galois representation $\rho_{\Pi, l}$ associated to $\Pi$ by V. Lafforgue, as well as the Frobenius-semisimplification of its local component $\rho_{\Pi, v_0, l}$. We have:
 \vskip 5pt
 
 \begin{itemize}
 \item In view of Lemma \ref{L:supp} and the properties of $\Pi$, the statements and proof of  Proposition \ref{P:deligne} continue to apply to $\rho_{\Pi, l}$. 
 The main point is that one has
 \begin{align}
  \mu(s,\pi\times \tau , \psi)  &= \mu(s,\Pi_{v_0} \times \tau, \psi) \notag \\
 &=   \gamma(s, \rho_{\Pi, v_0, l} \times \phi_{\tau}, \psi)  \cdot \gamma(-s,\rho_{\Pi, v_0, l}^{\vee} \times \phi^{\vee}_{\tau}, \psi) \cdot \gamma(2s, r_2 \circ \phi_{\tau}, \psi) \cdot \gamma(-2s, r_2^{\vee} \circ \phi_{\tau}, \psi). \notag \end{align}
 Here the first equality follows by Lemma \ref{L:supp} and the second follows by Corollary \ref{C:plan}.
  Thus, the proof of Proposition \ref{P:deligne}  shows that $\rho_{\Pi, v_0, l}$ corresponds to a discrete series $L$-parameter of $H_F$.
 
 \vskip 5pt
 
 \item Thanks to Lemma \ref{L:supp} again,  we deduce by the proof of Proposition \ref{P:indep} that $\rho_{\Pi, v_0, l}$ is independent of the choice of the prime $l$ or the globalization $\Pi$ used (as long as $\rho_{\Pi, l}$ is pure of weight $0$ and $\Pi_{v_0}$ has the same supercuspidal support as $\pi$). 
 \end{itemize}
 \vskip 5pt
 
 In view of the above, we may set 
 \[  \mathcal{L}(\pi)  = \text{the Frobenius-semisimplification of $\rho_{\Pi, v_0, l}$} \]
 where $\Pi$ is a cuspidal representation constructed as above. In this way, we have extended the map $\mathcal{L}$ of Theorem \ref{T:L} (except for Property (ii)) to all discrete series representations. Applying the Langlands classification theorem,  we then recover Theorem \ref{T:L2} (except for Property (iii)), albeit under the additional Working Hypothesis II. We do not get Property (iii) (for generic discrete series representations) this way because in using the trace formula to globalize,  we could not ensure that the globalization of a generic representation is globally generic.  
 \vskip 5pt
 We hope that the application discussed in this section will provide some impetus for a systematic development of the local theory of invariant harmonic analysis and the global theory of the Arthur-Selberg trace formula for general reductive groups over function fields of  characteristic $p>0$.

\end{document}